%% file: cartan-dimension-drop-algebra.tex
\documentclass[leqno,11pt]{article}

\usepackage[utf8]{inputenc}
\usepackage[T1]{fontenc}
\usepackage[english]{babel}
\usepackage[sfdefault,lf]{carlito}
\usepackage{fancyhdr}

\usepackage{hyperref}
\hypersetup{
  colorlinks   = true, %Colours links instead of ugly boxes
  urlcolor     = blue, %Colour for external hyperlinks
  linkcolor    = blue, %Colour of internal links
  citecolor   = black %Colour of citations
}

\usepackage[inline]{enumitem}

\usepackage{graphicx}

\usepackage{color}

\usepackage[intlimits]{amsmath}
\usepackage{amssymb, amsfonts}
\usepackage[thmmarks, amsmath, amsthm]{ntheorem}

\usepackage{mathabx}
\usepackage{MnSymbol}
\usepackage{mathbbol}

\usepackage{bbm}

\usepackage{mathrsfs}

\usepackage[all]{xy}

\numberwithin{equation}{section}
\newtheorem{theoremcounter}{theoremcounter}[section]

\theoremstyle{plain}

\newtheorem{corollary}[theoremcounter]{Corollary}
\newtheorem{lemma}[theoremcounter]{Lemma}
\newtheorem{proposition}[theoremcounter]{Proposition}

\newtheorem{theorem}[theoremcounter]{Theorem}

\theoremnumbering{Alph}
\newtheorem{introtheorem}{Theorem}
\theoremstyle{definition}

\newtheorem{definition}[theoremcounter]{Definition}

\theoremstyle{remark}
\newtheorem{example}[theoremcounter]{Example}

\newtheorem{notation}[theoremcounter]{Notation}

\newtheorem{remark}[theoremcounter]{Remark}
\newtheorem*{introquestion}{Question}

\input{shortcuts.tex}

\newcommand{\lwreath}{\mathrel{\ensuremath{ \text{\reflectbox{$\wreath$}}}}}

\renewcommand{\Sym}{\ensuremath{\mathrm{Sym}}}
\newcommand{\ul}{\ensuremath{\underline}}

\newcommand{\authors}{Sel{\c c}uk Barlak and Sven Raum}
\renewcommand{\title}{Cartan subalgebras in dimension drop algebras}
\newcommand{\shorttitle}{Cartan subalgebras in dimension drop algebras}

\DeclareFontEncoding{U}{}{}
\DeclareFontFamily{U}{bbold}{}
\DeclareFontShape{U}{bbold}{m}{n}
 {  <5> <6> <7> <8> <9> gen * bbold
   <10> <10.95> bbold10
  <12> <14.4> bbold12
 <17.28> <20.74> <24.88> bbold17
  }{}
\DeclareSymbolFont{bbold}{U}{bbold}{m}{n}
\DeclareSymbolFontAlphabet{\mathbbold}{bbold}

\begin{document}

%%%%%%%%
%% Front page
%%%%%%%%

\thispagestyle{empty}

\begin{center}
  \begin{minipage}[c]{0.9\linewidth}
    \textbf{\LARGE \title} \\[0.5em]
    by \authors
  \end{minipage}
\end{center}
  
\vspace{1em}

\renewcommand{\thefootnote}{}
\footnotetext{
  last modified on \today
}

\begin{center}
  \begin{minipage}{0.8\linewidth}
    \textbf{Abstract}.
    We completely classify Cartan subalgebras of dimension drop algebras with coprime parameters.  More generally, we classify Cartan subalgebras of arbitrary stabilised dimension drop algebras that are non-degenerate in the sense that the dimensions of their fibres in the endpoints are maximal. Conjugacy classes by an automorphism are parametrised by certain congruence classes of matrices over the natural numbers with prescribed row and column sums. In particular, each dimension drop algebra admits only finitely many non-degenerate Cartan subalgebras up to conjugacy. As a consequence of this parametrisation, we can provide examples of subhomogeneous \Cstar-algebras with exactly $n$ Cartan subalgebras up to conjugacy. Moreover, we show that in many dimension drop algebras two Cartan subalgebras are conjugate if and only if their spectrum is homeomorphic. 

  \end{minipage}
\end{center}

%%%%%%%%%
%% Document body
%%%%%%%%%

\input{introduction.tex}
\input{preliminaries.tex}
\input{one-sided.tex}
\input{non-degenerate.tex}
\input{twisted.tex}
\input{counting.tex}
\input{homeomorphism.tex}

%%%%%%%%%%
%% Bibliography
%%%%%%%%%%

% styles plain, mybibtexstyle
 %\bibliographystyle{mybibtexstyle}
% \bibliographystyle{abbrv}

% \printbibliography

%%%%%%%%%%
%% Address
%%%%%%%%%%

\vspace{2em}
{\small \parbox[t]{250pt}
  {
    Sel{\c c}uk Barlak \\
    E-Mail \textrm{selcuk@barlak.de}
  }
}
{\small \parbox[t]{200pt}
  {
    Sven Raum \\
    EPFL SB SMA \\
    Station 8 \\
    CH-1015 Lausanne \\
    E-Mail \textrm{sven.raum@epfl.ch}
  }
}

\end{document}

%% file: shortcuts.tex
\newcommand{\cB}{\ensuremath{\mathcal{B}}}

\newcommand{\cN}{\ensuremath{\mathcal{N}}}

\newcommand{\cP}{\ensuremath{\mathcal{P}}}

\newcommand{\cU}{\ensuremath{\mathcal{U}}}

\newcommand{\bC}{\ensuremath{\mathbb{C}}}

\newcommand{\bT}{\ensuremath{\mathbb{T}}}

\newcommand{\rD}{\ensuremath{\mathrm{D}}}
\newcommand{\rE}{\ensuremath{\mathrm{E}}}
\newcommand{\rF}{\ensuremath{\mathrm{F}}}

\newcommand{\rM}{\ensuremath{\mathrm{M}}}
\newcommand{\rN}{\ensuremath{\mathrm{N}}}

\newcommand{\rV}{\ensuremath{\mathrm{V}}}

\newcommand{\rmc}{\ensuremath{\mathrm{c}}}

\newcommand{\rmr}{\ensuremath{\mathrm{r}}}

\newcommand{\rmt}{\ensuremath{\mathrm{t}}}

\newcommand{\veps}{\ensuremath{\varepsilon}}
\newcommand{\eps}{\ensuremath{\epsilon}}
\newcommand{\vphi}{\ensuremath{\varphi}}

\newcommand{\ol}{\overline}

\newcommand{\eqstop}{\ensuremath{\, \text{.}}}
\newcommand{\eqcomma}{\ensuremath{\, \text{,}}}

\newcommand{\NN}{\ensuremath{\mathbb{N}}}
\newcommand{\ZZ}{\ensuremath{\mathbb{Z}}}

\newcommand{\CC}{\ensuremath{\mathbb{C}}}

\newcommand{\id}{\ensuremath{\mathrm{id}}}
\newcommand{\ra}{\ensuremath{\rightarrow}}

\newcommand{\thra}{\ensuremath{\twoheadrightarrow}}

\newcommand{\tr}{\ensuremath{\mathrm{tr}}}
\newcommand{\Aut}{\ensuremath{\mathrm{Aut}}}
\newcommand{\bs}{\ensuremath{\backslash}}

\newcommand{\ot}{\ensuremath{\otimes}}
\newcommand{\Cstar}{\ensuremath{\mathrm{C}^*}}

\newcommand{\Wstar}{\ensuremath{\mathrm{W}^*}}

\newcommand{\bo}{\ensuremath{\mathcal{B}}}
\newcommand{\cont}{\ensuremath{\mathrm{C}}}
\newcommand{\contb}{\ensuremath{\mathrm{C}_\mathrm{b}}}
\newcommand{\conto}{\ensuremath{\mathrm{C}_0}}
\newcommand{\lspan}{\ensuremath{\mathop{\mathrm{span}}}}
\newcommand{\Ad}{\ensuremath{\mathop{\mathrm{Ad}}}}
\newcommand{\Sym}{\ensuremath{\operatorname{Sym}}}

%% file: introduction.tex
\section{Introduction}
\label{sec:introduction}

Cartan subalgebras constitute a centrepiece of modern structure theory of von Neumann algebras.  They were introduced by Dixmier in \cite{dixmier53-sous-anneaux-abliens} and related to measurable group theory and ergodic theory by Singer in \cite{singer55}.  Thanks to Popa's intertwining by bimodule techniques  introduced in \cite[Theorem A.1]{popa06-betti-numbers} and \cite[Theorem 2.1]{popa06-strong-rigidity-1}, which are particularly well compatible with Cartan subalgebras, best possible classification results for Cartan subalgebras in certain crossed product von Neumann algebras could be obtained \cite{popavaes11_2,popavaes12}.

\Cstar-algebraic Cartan subalgebras find their origin in the notion of groupoid \Cstar-algebras \cite{renault80} and were introduced much later than their von Neumann algebraic counterparts in \cite{renault08-cartan} by Renault, building on Kumjian's work on \Cstar-diagonals \cite{kumjian86}.  In analogy with Feldman-Moore's work on measured equivalence relations and von Neumann algebraic Cartan subalgebras \cite{feldmanmoore77,feldmanmoore77_1}, Cartan subalgebras of \Cstar-algebras connect to topological dynamics \cite{li15-rigidity} and geometric group theory \cite{spakulawillett11,li16-quasi-isometry} and therefore provide structure that has similar potential as Cartan subalgebras in von Neumann algebras.  Besides classical applications of groupoid \Cstar-algebras, which take a slightly different perspective, the potential of Cartan subalgebras in \Cstar-algebras is visible through a characterisation of a positive solution of the infamous UCT problem in \cite{barlakli15-uct,barlakli17-uct}. Moreover, (non-)existence and (non-)uniqueness results are discussed in \cite{lirenault17}, where among others a classification of Cartan subalgebras in homogeneous \Cstar-algebras in terms of principal bundles is provided. Recently in \cite{kundbyraumthielwhite17} yet another connection with \Cstar-algebraic Cartan subalgebras was found.  We say that two Cartan subalgebras in a \Cstar-algebra $B_1, B_2 \subset A$ are conjugate, if there is an automorphism $\alpha \in \Aut(A)$ such that $\alpha(B_1) = B_2$.  The \Cstar-superrigidity problem for torsion-free virtually abelian groups - i.e.\ to recover the group from its group \Cstar-algebras - can be solved assuming a classification up to conjugacy for Cartan subalgebras in certain subhomogeneous \Cstar-algebras.  This parallels the importance of classification results for Cartan subalgebras in proofs of \Wstar-superrigidity \cite{ioanapopavaes10,berbecvaes12}, although taking very different forms.  Work in \cite{kundbyraumthielwhite17} strongly motivates us to study Cartan subalgebras in subhomogeneous \Cstar-algebras. 

%Some classification results for Cartan subalgebras in \Cstar-algebras were obtained in \cite{lirenault17}, although they were deduced from known results about principle bundles (working in homogeneous \Cstar-algebras) and from von Neumann algebraic results (working in the reduced group \Cstar-algebra of the free groups), respectively. 

Further motivation for the study of Cartan subalgebras in \Cstar-algebras comes from the recent breakthrough results in the structure and classification theory of simple nuclear \Cstar-algebras achieved by many hands; see among others \cite{matuisato2012, matuisato2014, elliottgonglinniu2015, gonglinniu15, satowhitewinter2015, tikuisiswhitewinter17}. Many simple nuclear \Cstar-algebras that are classifiable in the sense of the Elliott program are known to have Cartan subalgebras, although it remains an open problem whether this is always the case. It is indeed true for UCT Kirchberg algebras, which follows either by work of Spielberg \cite{spielberg07} or by combining results of Katsura \cite{katsura2008} and Yeend \cite{yeend06, yeend07}. On the other hand, many stably finite nuclear classifiable \Cstar-algebras also admit Cartan subalgebras; see for example \cite{renault80, deeleyputnamstrung2015, putnam2016}.    \Cstar-algebras in the latter class are often expressible as inductive limits of subhomogeneous \Cstar-algebras, so that one might expect many of their Cartan subalgebras to arise as inductive limits of Cartan subalgebras of subhomogeneous \Cstar-algebras as well. This in particular applies to the Jiang-Su algebra $\mathcal Z$, \cite{jiangsu99}, which plays a key role in the structure and classification theory of simple nuclear \Cstar-algebras and which can be constructed as an inductive limit of prime dimension drop algebras.  

The purpose of this article to understand Cartan subalgebras of stabilised dimension drop algebras up to conjugacy by an automorphism. For this, we take a purely \Cstar-algebraic approach and classify a large class of Cartan subalgebras of dimension drop algebras and their stabilisations.  Dimension drop algebras are arguably among the most simple and at the same time most important subhomogeneous \Cstar-algebras, so that they form the ideal starting point for a systematic study of Cartan subalgebras in general subhomogeneous \Cstar-algebras.  We will adopt the following notation from \cite{jiangsu99} for stabilised dimension drop algebras throughout the article, denoting the algebra of matrices of size $m \times m$ by $\rM_m$,
\begin{equation*}
  I_{m,n,o}
  =
  \{
  f \in \cont([0,1], \rM_m \ot \rM_n \ot \rM_o) \mid
  f(0) \in \rM_m \ot 1 \ot \rM_o,  f(1) \in 1 \ot \rM_n \ot \rM_o
  \}
  \eqstop
\end{equation*}
We also write $I_{m,n} = I_{m,n,1}$ for the non-stabilised dimension drop algebras.

In this paper, we restrict our attention to what we term non-degenerate Cartan subalgebras of stabilised dimension drop algebras.  A Cartan subalgebra $B \subset I_{m,n,o}$ is called non-degenerate if the fibres at the endpoints satisfy $\dim B_0 = mo$ and $\dim B_1 = no$ and we call it degenerate otherwise, cf. Definition~\ref{def:non-degenerate}.  G{\'a}bor Szab{\'o} kindly pointed out to us that these are exactly the \Cstar-diagonals in the sense of Kumjian \cite{kumjian86} inside $I_{m,n,o}$, which are precisely the Cartan subalgebras with the unique extension property in the sense of Anderson \cite{anderson79}.  While Example \ref{ex:completely-degenerate-cartan} demonstrates the existence of degenerate Cartan subalgebras in dimension drop algebras of the form $I_{m,m}$, it turns out that Cartan subalgebras of the most important class of dimension drop algebras are all non-degenerate.  In view of Proposition \ref{prop:cstar-diagonals}, the next theorem answers \cite[Problem 7]{barlakszabo17-problems} in the setting of stabilised dimension drop algebras.
\begin{introtheorem}[See Theorem \ref{thm:cartans-non-degenerate}]
\label{intro:thm:non-degenerate Cartan subalgebras}
  The following two statements are equivalent for a stabilised dimension drop algebra $I_{m,n,o}$.
  \begin{itemize}
  \item $m,n,o$ are pairwise coprime.
  \item Every Cartan subalgebra in $I_{m,n,o}$ is non-degenerate.
  \item Every Cartan subalgebra of $I_{m,n,o}$ is a \Cstar-diagonal, that is, it has the unique extension property.
  \end{itemize}
\end{introtheorem}
 Our first main result provides a parametrisation of conjugacy classes of non-degenerate Cartan subalgebras in terms of classical combinatorial objects explained below.  We emphasise that we classify Cartan subalgebras up to conjugacy by an automorphism.  This is in contrast to the usual conjugation by unitary results in von Neumann algebras -- see \cite{speelmanvaes12, spaas17} for an account on the model theoretic complexity of 'conjugacy by an automorphism' and 'conjugacy by a unitary'.  As a consequence, our strategy for the classification in Theorem \ref{intro:thm:classificaiton} is substantially different from the strategy employed in a von Neumann algebraic setup.  We first provide a list of Cartan subalgebras in dimension drop algebras which are explicitly described.  Within this class we provide a classification result.  Only then we prove that every non-degenerate Cartan subalgebra of a dimension drop algebra is conjugate to one coming from our list.
\begin{introtheorem}[See Corollary \ref{thm:classification-by-matrices}]
  \label{intro:thm:classificaiton}
    Conjugacy classes of non-degenerate Cartan subalgebras in $I_{m,n,o}$ are parametrised by congruence classes of matrices in $\rM(mo, n, no, m)$.
\end{introtheorem}
Here $\rM(a,b,c,d)$ denotes the set of matrices of size $a \times c$ with entries in the natural numbers such that each of the $a$ rows sums to $b$ and each of the $c$ columns sums to $d$.  Note that in order to obtain a non-empty set, the compatibility condition $ab = cd$ is required, see Notation \ref{not:row-column-sum-fixed}.   Two matrices $A, B \in \rM_{m,n}(\NN)$ are congruent if there are permutation matrices $\rho_1 \in \Sym(m)$ and $\rho_2 \in \Sym(n)$ such that $A  = \rho_1 B \rho_2$ or, in case $m = n$, $A^\rmt = \rho_1 B \rho_2$, see Definition \ref{def:congruence}.

Thanks to the parametrisation provided by Theorem \ref{intro:thm:classificaiton}, we are able to count Cartan subalgebras in some families of dimension drop algebras.  This allows us to solve the \Cstar-algebraic analogue of a well-known open problem in von Neumann algebras, namely to find for each $n \in \NN$ some ${\rm II}_1$ factor with exactly $n$ Cartan subalgebras up to (unitary) conjugacy.  In the von Neumann algebraic setting only partial results addressing this problem are available \cite{connesjones85,ozawapopa10-cartan1,popavaes10-superrigidity,speelmanvaes12,krogagervaes15}.  In the \Cstar-algebraic context, Li-Renault \cite{lirenault17} translate conjugacy of Cartan subalgebras of homogeneous \Cstar-algebras into a problem about principal bundles and then make use of known results on principle bundles to provide examples of homogeneous \Cstar-algebras with exactly $p(n)$ Cartan subalgebras up to conjugacy for any $n \in \NN_{\geq 1}$, where $p(n)$ denotes the number of partitions of $n$.
\begin{introtheorem}[See Corollary \ref{cor:exactly-n-cartans}]
  \label{intro:thm:exactly-n-cartans}
  For every $n \in \NN$ there is a subhomogeneous \Cstar-algebra that has exactly $n$ Cartan subalgebras up to conjugacy.
\end{introtheorem}
In Remark \ref{rem:infinitely-many-cartans}, we show that Theorem \ref{intro:thm:exactly-n-cartans} also allows one to construct \Cstar-algebras with exactly continuum many Cartan subalgebras up to conjugacy.

% \comment{add explanations here}

% \begin{introtheorem}[See Theorem \ref{thm:estimata-cartan-subalgebras}]
%   \label{intro:thm:many-cartans}
%   \comment{add a statement showing that there are many many Cartan subalgebras in big $I_{m,n,o}$}
% \end{introtheorem}

Comparing von Neumann algebraic with \Cstar-algebraic Cartan subalgebras, it becomes apparent that in the latter context there is one fundamental obstruction to conjugacy.  While all separable abelian diffuse von Neumann algebras are pairwise isomorphic, separable abelian \Cstar-algebras are classified by the homeomorphism type of their spectrum.  In \cite{kundbyraumthielwhite17}, it was already observed that plain uniqueness of Cartan subalgebra results are too much to be expected in the setting provided by \Cstar-superrigidity of virtually abelian groups.  The second most optimistic approach tries to prove that two Cartan subalgebras of a subhomogeneous \Cstar-algebra are conjugate if and only if their spectra are homeomorphic - this statement would suffice to prove \Cstar-superrigidity of virtually abelian groups based on results in \cite{kundbyraumthielwhite17}.  Surprisingly, this statement holds true in all dimension drop algebras and most stabilised dimension drop algebras, when one restricts to non-degenerate Cartan subalgebras. 
\begin{introtheorem}[See Theorem \ref{thm:conjugacy-homeomorphism}]
  \label{intro:thm:conjugacy-homeomorphism}
  Let $I_{m,n,o}$ be a stabilised dimension drop algebra such that either $(m,n) \neq (2,2)$ or $o = 1$.  Then two non-degenerate Cartan subalgebras of $I_{m,n,o}$ are conjugate by an automorphism if and only if their spectra are homeomorphic.
\end{introtheorem}
In Proposition \ref{prop:no-conjugacy-homeomorphism}, we prove that the excluded cases $I_{2,2,o}$, for $o \geq 2$, do not obey the conclusion of Theorem \ref{intro:thm:conjugacy-homeomorphism}, hence rendering a general classification of Cartan subalgebras by their spectrum too optimistic.  Nevertheless, Theorem \ref{intro:thm:conjugacy-homeomorphism} is enough evidence to justify the following question.
\begin{introquestion}
  In which subhomogeneous \Cstar-algebras are two Cartan algebras conjugate if and only if their spectra are homeomorphic?
\end{introquestion}

\noindent This article has six sections.  After the introduction and some preliminaries, we study in Section~\ref{sec:one-sided} Cartan subalgebras of one-sided dimension drop algebras $I_{1, n, m}$. It is at this point where we introduce the notion of non-degenerate Cartan subalgebras in stabilised dimension drop algebras, and then prove that one-sided dimension drop algebras have a unique non-degenerate Cartan subalgebra up to conjugacy. This will be an important tool in the remaining sections. Section \ref{sec:non-degenerate-cartan} is devoted to the study of non-degenerate Cartan subalgebras in stabilised dimension drop algebras, which culminates in Theorem \ref{intro:thm:non-degenerate Cartan subalgebras}. In Section \ref{sec:twisted}, we introduce a class of Cartan subalgebras of dimension drop algebras that we term twisted standard Cartan subalgebras and classify them up to conjugacy. We conclude that section by proving that every non-degenerate Cartan subalgebra of a stabilised dimension drop algebra is conjugate to a twisted standard Cartan subalgebra.  In Section \ref{sec:counting}, we provide a parametrisation of conjugacy classes of twisted standard Cartan subalgebras by congruence classes of matrices as described above.   This leads to the proof of Theorem \ref{intro:thm:classificaiton}. Furthermore, it allows us to deduce some explicit counting results.  In particular, we obtain Theorem \ref{intro:thm:exactly-n-cartans}.  In Section \ref{sec:homeomorphism}, we study the spectra of twisted standard Cartan subalgebras and prove Theorem \ref{intro:thm:conjugacy-homeomorphism}.

\subsection*{Acknowledgements}

S.B. is supported by the Villum Fonden project grant ‘Local and global structures of groups and
their algebras’ (7423) (2014–2018).  Part of the work on this article was done during S.B.'s visit to EPFL and S.R.'s visit to the University of Southern Denmark.  We thank S{\o}ren Knudby for pointing out a mistake in Remark \ref{rem:infinitely-many-cartans}.

%%% Local Variables:
%%% mode: latex
%%% TeX-master: "cartan-dimension-drop-algebra"
%%% End:

%% file: preliminaries.tex
\section{Preliminaries}
\label{sec:preliminaries}

In this section, we fix some notation used throughout this work and recall some important definitions. For a natural number $n \geq 1$, write  $\ul n = \{1, \dotsc, n\}$. If $X$ is any set, we denote by $\Sym(X)$ the set of all bijections of $X$.  For short, $\Sym(n) = \Sym(\ul n)$.  Identifying a permutation $\sigma \in \Sym(n)$ with its permutation matrix satisfying
\begin{equation*}
  \sigma_{i,i'} =  \delta_{\sigma(i'), i} = 
  \begin{cases}
   1, & \text{if } \sigma(i') = i, \\
   0, & \text{otherwise,}
  \end{cases}
\end{equation*}
we obtain an embedding $\Sym(n) \ra \cU(n)$. We moreover denote by $e_{ij} \in \rM_n$ the matrix whose $ij$-th entry is $1$ and whose other entries are all zero. Similarly, $e_i \in \bC$ denotes the $i$-th standard vector.

Given an inclusion of \Cstar-algebras $B \subset A$, the normaliser of $B$ in $A$ is
\begin{equation*}
  \rN_A(B) = \{x \in A \mid x B x^*, x^* B x \subset B\}
  \eqstop
\end{equation*}
We call $B$ regular in $A$, if $\rN_A(B)$ generates $A$ as a \Cstar-algebra.  Moreover, if the inclusion $B \subset A$ is unital, we call
\begin{equation*}
  \cN_B(A) = \{u \in \cU(A) \mid uBu^* = B\}  
\end{equation*}
the unitary normaliser of $B$ in $A$. The normaliser $N_A(B)$ is a closed subset of $A$ that is also closed under multiplication and the $*$-operation, and the unitary normaliser $\cN_A(B)$ is a group, if it is defined.  In particular, if $B \subset A$ is regular, then $\lspan N_A(B)$ is dense in $A$. 

We also denote by $\Aut(B \subset A)$ the set of all $*$-automorphisms of $A$ that preserve $B$ setwise, that is, $\alpha \in \Aut(A)$ belongs to $\Aut(B \subset A)$ exactly if $\alpha(B) = B$. Furthermore, we call $B \subset A$ a MASA if $B$ is a maximal abelian \Cstar-subalgebra of $A$.

\begin{definition}[cf.\ {\cite[Definition~5.1]{renault08-cartan}}]
  \label{def:cartan-subalgebra}
Let $A$ be a \Cstar-algebra. A MASA $B \subset A$ is said to be a Cartan subalgebra if
\begin{enumerate}
\item $B$ contains an approximate unit for $A$,
\item $B$ is regular, and
\item there exists a faithful conditional expectation $A \to B$.
\end{enumerate}
In this case, $(A,B)$ is called a Cartan pair.
\end{definition}

\begin{example}
  \label{ex:cartan-finite-dimensional}
  For $m \geq 2$, the maximal abelian \Cstar-subalgebra of diagonal matrices $\rD_m \subset \rM_m$ is a Cartan subalgebra. More concretely, an element $a = \sum_{i,j = 1}^m \lambda_{i,j} e_{ij}$ normalises $\rD_m$ if and only if for each $i$, $\lambda_{i,j} \neq 0$ for at most one $j$ and for each $j$, $\lambda_{i,j} \neq 0$ for at most one $i$. In other words, $\rN_{\rM_m}(\rD_m) = \rD_m \rtimes \Sym(m)$, where $\Sym(m)$ acts by permutation  matrices. The unitary normaliser therefore satisfies $\cN_{\rM_m}(\rD_m) = \bT^m \rtimes \Sym(m)$. It follows that the set $\Aut(\rD_m \subset \rM_m)$ of automorphisms of $\rM_m$ that preserve $\rD_m$ is isomorphic with $(\bT^m/\bT) \rtimes \Sym(m)$. Up to (inner) conjugacy, $\rD_m$ is the unique Cartan subalgebra of $\rM_m$.

%  Since the inclusion $\rD_m \otimes \rD_n \subset \rM_m \otimes \rM_n$ is isomorphic with $\rD_{mn} \subset \rM_{mn}$, the previous paragraph applies to describe its normaliser and its unitary normaliser.  We make the following notational distinction, and write
%  \begin{equation*}
%    \cN_{\rM_m \otimes \rM_n}(\rD_m \otimes \rD_n)
%    =
%    (\bT^m \ot \bT^n) \rtimes \Sym(\ul m \times \ul n)
%  \end{equation*}
%  and
%  \begin{equation*}
%    \Aut(\rD_m \ot \rD_n \subset \rM_m \ot \rM_n)
%    =
%    (\bT^m \ot \bT^n) / \bT \rtimes \Sym(\ul m \times \ul n)
%    \eqstop
%  \end{equation*}
%  In particular, if $\tr:\rM_n \to \CC$ denotes the normalised trace, then $(\id \ot \tr)(\rN_{\rM_m \ot \rM_n}(\rD_m \ot \rD_n))  \subset  \rN_{\rM_m}(\rD_m) \ot 1$.
\end{example}

Recall that a \Cstar-algebra $A$ is called $n$-subhomogeneous for $n \in \NN$, if all of its irreducible representations have dimension at most $n$.  We say that $A$ is subhomogeneous if it is $n$-subhomogeneous for some $n \in \NN$. Furthermore, $A$ is called homogeneous if there exists some $n \in \NN$ such that all irreducible representations of $A$ have dimension $n$. Subhomogeneous \Cstar-algebras are exactly the \Cstar-subalgebras of homogeneous \Cstar-algebras; see \cite[IV.1.4.3]{blackadar06}.

\begin{definition}[cf.\ \cite{jiangsu99}]
  \label{def:dimension-drop-algebra}
The dimension drop algebra with parameters $m, n \in \NN_{\geq 1}$ is
\begin{equation*}
  I_{m,n}
  =
  \{
  f \in \cont([0,1], \rM_m \ot \rM_n) \mid
  f(0) \in \rM_m \ot 1,  f(1) \in 1 \ot \rM_n
  \}
  \eqstop
\end{equation*}
More generally, a stabilised dimension drop algebra is 
\begin{equation*}
  I_{m,n,o}
  =
  \{
  f \in \cont([0,1], \rM_m \ot \rM_n \ot \rM_o) \mid
  f(0) \in \rM_m \ot 1 \ot \rM_o,  f(1) \in 1 \ot \rM_n \ot \rM_o
  \}
  \eqcomma
\end{equation*}
for parameters $m,n,o \in \NN_{\geq 1}$.
\end{definition}

Let $X \subset \cont([0,1], M_n)$ be a subset. For a subinterval $I \subset [0,1]$, we define
\begin{equation*}
  X_I = \{f_{|I} \in \conto(I, M_n) \mid  f \in X,\ f_{|I} \in \conto(I, M_n) \}
  \eqstop
\end{equation*}
We also write $X_{\{t\}} = X_t$. For $x \in X$, we denote by $x_t = x(t) \in X_t$.

We conclude this section with a few words about Cartan subalgebras in homogeneous \Cstar-algebras over an interval. In \cite{lirenault17}, the relation between Cartan subalgebras in homogeneous \Cstar-algebras over a space $X$ and principal $\Aut(\rD_n \subset \rM_n)$-bundles over $X$ was pointed out.  In particular, in homogeneous \Cstar-algebras over contractible spaces, there is a unique Cartan subalgebra.
\begin{theorem}[cf.\ {\cite[Section 2]{lirenault17}}]
  \label{thm:li-renault}
  Let $A$ be a homogeneous \Cstar-algebra over a contractible space.  Then $A$ admits a unique Cartan subalgebra up to conjugacy.
\end{theorem}

We will need a stronger result for the special case of homogeneous \Cstar-algebras over an interval, which immediately follows from Li-Renault's.
\begin{corollary}
  \label{cor:li-renault-unitary}
  Let $I$ be an interval and $A = \cont_0(I, \rM_m)$.  Then $A$ admits a unique Cartan subalgebra up to conjugacy by an inner automorphism associated with a unitary in the multiplier algebra of $A$, $M(A) = \contb(I,\rM_m)$.
\end{corollary}
\begin{proof}
  Let $B \subset A$ be a Cartan subalgebra.  Since the interval $I$ is contractible, we can apply Theorem \ref{thm:li-renault} to obtain an automorphism $\alpha \in \Aut(A)$ such that $\alpha(B)$ is the standard Cartan subalgebra $C = \cont_0(I, \rD_m)$.  Composing $\alpha$ with an automorphism in $\Aut(C \subset A)$, we may assume that $\alpha|_{\cont_0(I)} = \id$.  We can then consider $\alpha$ as a continuous map $I \ra \Aut(\rM_m) \cong \cU(m)/\bT$.  Let $(I_n)_{n \in \NN}$ be an ascending sequence of closed intervals in $I$ such that $\bigcup_n I_n = I$.   Since $\cU(m) \thra \cU(m)/\bT$ has the structure of a principal $\bT$-bundle, it has the homotopy lifting property, so that we find for every $n \in \NN$ a lift $u_n \in \cont(I_n, \cU(m))$ for $\alpha|_{I_n}$. In other words, $\alpha|_{I_n} = \Ad u_n$. Now define $v_1 = u_1$.  If $v_n$ is defined, consider $f_n := (u_{n+1}|_{I_n})^*v_n \in \cont(I_n, \bT)$.  Extend $f_n$ to an element $g_n \in \cont(I_{n+1}, \bT)$ and set $v_{n+1} = u_{n+1}g_n$.  We obtain a family $(v_n)_n$ of unitaries with $v_n \in \cont(I_n, \cU(m))$ satisfying $v_{n+1}|_{I_n} = v_n$ and $\Ad v_n = \alpha|_{I_n} \in \Aut(\cont(I_n,\rD_m))$.  We can now define $v \in \cont_b(I,\rM_m)$ by the requirement $v|_{I_n} = v_n$ for all $n \in \NN$ and obtain a unitary satisfying $\Ad v = \alpha$.  This proves that $B$ is unitarily conjugate to the standard Cartan subalgebra $C \subset A$ via the unitary $v \in M(A)$.
\end{proof}

%%% Local Variables:
%%% mode: latex
%%% TeX-master: "cartan-dimension-drop-algebra"
%%% End:

%% file: one-sided.tex
\section{One-sided dimension drop algebras}
\label{sec:one-sided}

In this section, we are going to investigate uniqueness of Cartan subalgebras in one of the easiest possible subhomogeneous \Cstar-algebras, which we call here one-sided dimension drop algebras.
\begin{definition}
  \label{def:one-sided-dimension-drop}
  The one-sided dimension drop algebra for $m,n \geq 1$ is
  \begin{equation*}
    J_{m,n} = \{f \in \cont([0,1], \rM_m \ot \rM_n) \mid f(0) \in \rM_m \ot 1\}
    \eqstop
  \end{equation*}
\end{definition}

\begin{remark}
  \label{rem:one-sided-is-stabilised-dimension-drop}
  The one-sided dimension drop algebras are all stabilised dimension drop algebras.  We have $J_{m,n} = I_{1,n,m}$.
\end{remark}

\begin{example}
  \label{ex:cartan-one-sided-dimension-drop}
  We are going to call the subalgebra
  \begin{equation*}
    C_{m,n} = \{f \in \cont([0,1], \rD_m \ot \rD_n) \mid f(0) \in \rD_m \ot 1\} \subset J_{m,n}
  \end{equation*}
  the standard Cartan subalgebra of $J_{m,n}$.  It is a Cartan subalgebra indeed, as the following argument shows.

  First we show that $C_{m,n} \subset J_{m,n}$ is a MASA.  This follows from the fact that for every $t \in (0,1]$, the fibre $\rD_m \ot \rD_n = (C_{m,n})_t \subset (J_{m,n})_t = \rM_m \ot \rM_n$ is a MASA.  So if $f \in C_{m,n}' \cap J_{m,n}$, then $f(t) \in \rD_m \ot \rD_n$ for all $t \in (0,1]$ and hence $f(0) \in \rD_m \ot \rD_n \cap \rM_m \ot 1 = \rD_m \ot 1$.  This shows that $f \in C_{m,n}$.  Second, the natural conditional expectation $\cont([0,1], \rM_m \ot \rM_n) \ra \cont([0,1], \rD_m \ot \rD_n)$ restricts to a faithful conditional expectation $J_{m,n} \ra C_{m,n}$.  It remains to show that $C_{m,n}$ is regular in $J_{m,n}$.

  Let $f \in J_{m,n}$ and $\veps > 0$.  Since $f$ is continuous, there is $\delta > 0$ such that $\|f(t) - f(0)\| < \veps$ for all $t \in [0, \delta]$.  Because $(C_{m,n})_{[\delta, 1]} \subset (J_{m,n})_{[\delta, 1]}$ coincides with $C([\delta,1], \rD_m \ot \rD_n) \subset \cont([\delta, 1], \rM_m \ot \rM_n)$, there are finitely many $\tilde f_i \in N_{(J_{m,n})_{[\delta, 1]}}((C_{m,n})_{[\delta, 1]})$ and numbers $c_i \in \CC$ such that $\|\sum_i c_i \tilde f_i - f|_{[\delta, 1]}\| < \veps$. In fact, we may choose the $\tilde f_i$ to be in $C([\delta,1], N_{\rM_m}(\rD_m) \odot N_{\rM_n}(\rD_n))$, where $N_{\rM_m}(\rD_m) \odot N_{\rM_n}(\rD_n)$ denotes the set of elementary tensors. Now, extend each $\tilde f_i$ to the unique element $f_i \in J_{m,n}$ that is affine on $[0,\delta]$ and satisfies $f_i(0) = (\id \ot \tr)(\tilde f_i(\delta))$.  A short calculation shows that $\|f - \sum_i c_i f_i\| < 4 \eps$. Since $(\id \ot \tr)(N_{\rM_m}(\rD_m) \odot N_{\rM_n}(\rD_n))  \subset  \rN_{\rM_m}(\rD_m) \ot 1$, it follows that $f_i(t)$ normalises $(C_{m,n})_t$ for all $t \in [0, 1]$. We conclude that each $f_i \in \rN_{J_{m,n}}(C_{m,n})$, which finishes the proof that $C_{m,n}$ is a Cartan subalgebra of $J_{m,n}$.
\end{example}

We first prove a uniqueness of Cartan subalgebras result in one-sided dimension drop algebras of the form $J_{1,n}$. For this, we introduce the following notation.  If $u \in \cont_b((0,1], \rM_n)$ is a unitary, then $\Ad u$ induces a unique automorphism of $J_{1,n}$.  Indeed, if $f \in J_{1,n}$ and $u \in \cont_b((0,1], \cU(n))$, then $f(t) \ra \lambda \in \bC$ as $t \ra 0$ implies that also $u_t f(t) u_t^* \ra \lambda$ as $t \ra 1$.
\begin{lemma}
  \label{lem:unique-masa-one-sided-dimension-drop}
  For every Cartan subalgebra $B \subset J_{1,n}$, there is a unitary $u \in \cont_b((0,1], \cU(n))$ such that $(\Ad u)(B) = C_{1,n}$.  In particular, $J_{1,n}$ has a unique Cartan subalgebra up to conjugacy by an automorphism.
\end{lemma}
\begin{proof}
  First note that $B_{(0,1]} \subset J_{1,n}|_{(0,1]} = \cont((0,1], \rM_n)$ is a Cartan subalgebra.  Hence, Corollary~\ref{cor:li-renault-unitary} provides us with a unitary $u \in \cont_b((0,1], \rM_n)$ such that $u B_{(0,1]} u^* = \conto( (0,1], \rD_n)$.  As discussed above, $\Ad u$ induces a unique automorphism of  $J_{1,n}$, which then satisfies $(\Ad u)(B) = C_{1,n}$.
\end{proof}

We would like to reduce considerations about Cartan subalgebras in general one-sided dimension drop algebras, to the case considered in Lemma \ref{lem:unique-masa-one-sided-dimension-drop}.   We can do so for the class of so-called non-degenerate Cartan subalgebras in stabilised dimension drop algebras described in the next definition, which will be treated in the rest of the article.  G{\'a}bor Szab{\'o} kindly pointed out to us that non-degenerate Cartan subalgebras in $I_{m,n,o}$ are exactly the \Cstar-diagonals in the sense of Kumjian \cite{kumjian86}, which are precisely the Cartan subalgebras with the unique extension property in the sense of Anderson \cite{anderson79}.  In stabilised dimension drop algebras with coprime parameters, every Cartan subalgebra is non-degenerate, as we will see in Theorem \ref{thm:cartans-non-degenerate}.
 \begin{definition}
   \label{def:non-degenerate}
   A Cartan subalgebra $B \subset J_{m,n}$ is called non-degenerate, if $\dim B_0 = m$.  In analogy, a Cartan subalgebra $B \subset I_{m,n,o}$ is called non-degenerate, if $\dim B_0 = mo$ and $\dim B_1 = no$.
 \end{definition}

Before proving uniqueness of non-degenerate Cartan subalgebras in arbitrary stabilised dimension drop algebras, let us analyse their normalisers.
\begin{lemma}
  \label{lem:partial-isometries-in-normaliser}
  Let $B \subset J_{m,n}$ be a Cartan subalgebra containing the constant functions $e_{ii} \ot 1$.  Then there exists $\veps > 0$ such that $e_{ij} \ot 1 \in \rN_{J_{m,n}}(B)|_{[0,\veps]}$ for all $i,j \in \{1, \dotsc, m\}$.
\end{lemma}
\begin{proof}
  Let $i,j \in \{1, \dotsm, m\}$ and $i \neq j$.  Since $\lspan \rN_{J_{m,n}}(B) \subset J_{m,n}$ is dense, $\lspan (e_{ii} \ot 1) \rN_{J_{m,n}} (B) (e_{jj} \ot 1)$ is dense in $(e_{ii} \ot 1) J_{m,n} (e_{jj} \ot 1)$.  It follows that 
  \begin{equation*}
    \{x_0 ~|~ x \in \lspan (e_{ii} \ot 1)\rN_{J_{m,n}}(B)(e_{jj} \ot 1)\} =  \CC (e_{ij} \ot 1)
    \eqstop
  \end{equation*}
  So there is some element $x \in (e_{ii} \ot 1)\rN_{J_{m,n}}(B)(e_{jj} \ot 1)$ such that $x_0 = c (e_{ij} \ot 1)$ for some $c \neq 0$.  Since $\rN_{J_{m,n}}(B)$ is closed under scalar multiplication, we can assume that $c = 1$.  

  Since $e_{ii} \ot 1, e_{jj} \ot 1 \in B \subset \rN_{J_{m,n}}(B)$, we actually have $x \in \rN_{J_{m,n}}(B)$.  In particular, $x^*x \in B $ and hence $|x| \in B$. Note that actually $|x| \in (e_{jj} \ot 1) B (e_{jj} \ot 1)$.  By continuity, we find some $\veps > 0$ such that $\| |x|_t - (e_{jj} \ot 1) \| < 1/2$ for all $t \in [0,2\veps]$.  Let $f \in \cont([0,1])$ be a function satisfying $f|_{[0,\veps]} \equiv 1$ and $f|_{[2\veps, 1]} \equiv 0$, and consider $v = x |x|^{-1}f \in \rN_{J_{m,n}}(B)$, where $|x|^{-1} \in (e_{jj} \ot 1) B (e_{jj} \ot 1)$ is the inverse of $|x|$ on $[0,2\veps]$ and the expression $|x|^{-1}f$ is understood to be equal to $0$ on $[2\veps, 1]$.   We obtain that for all $t \in [0,\veps]$, $v_t$ is a partial isometry in $(e_{ii} \ot 1) (\rM_m \ot \rM_n) (e_{jj} \ot 1)$ such that $v_t^*v_t = e_{jj} \ot 1$.  It follows that $v_t v_t^* = e_{ii} \ot 1$, and hence there exists a unitary $u \in C([0,1])$ such that $v_t = u_t (e_{ij} \otimes 1)$ for all $t \in [0,\veps]$.  Since $B$ contains $\cont([0,1])$, also $u^*v \in \rN_{J_{m,n}}(B)$.  This finishes the proof of the lemma.
\end{proof}

 \begin{proposition}
  \label{prop:unique-cartan-one-sided-dimension-drop}
  The one-sided dimension drop algebras have a unique non-degenerate Cartan subalgebra up to conjugacy.
\end{proposition}
\begin{proof}
  Let $B \subset J_{m,n}$ be a non-degenerate Cartan subalgebra.  We find a unitary $u \in J_{m,n}$ such that $(u B u^*)_0 = \rD_m \ot 1$. By semiprojectivity of $\CC^m$, there exists $\veps \in (0,1/3]$ and projections $f_1, \dotsc, f_m \in (J_{m,n})_{[0,\veps]}$ such that $f_1(0), \dotsc, f_m(0)$ are the minimal projections of $(u B u^*)_0$ and $\|f_i(0) - f_i(t)\| < 1$ for all $i \in \{1, \dotsc, m\}$ and all $t \in [0, \veps]$. Then $f_i(0) f_i(t) f_i(0)$ is invertible in $f_i(0) (\rM_m \ot \rM_n) f_i(0)$ and hence the polar decomposition of the function $t \mapsto f_i(0) f_i(t)$, $t \in [0,\veps]$, provides us with partial isometries $v_1, \dotsc, v_m \in (J_{m,n})_{[0,\veps]}$ with support projection $f_i$ and range projection $\mathbb 1_{[0,\veps]} f_i(0)$.  The unitary $v = v_1 + \dotsm + v_m \in (J_{m,n})_{[0,\veps]}$ satisfies $v u B_{[0,\veps]} u^* v^* \subset \cont([0, \veps], \rD_m \ot 1)$.  We can extend $v$ to a unitary in $J_{m,n}$ such that $v u|_{[2\veps, 1]} = 1$.  Hence, from now on we may assume that there is $\veps \in (0, 1]$  such that $B_{[0,\veps]}$ contains the constant function with value $e_{ii} \ot 1$ for all $i \in \{1, \dotsc, m\}$.

As $(e_{11} \ot 1)(J_{m, n})_{[0, \veps]} (e_{11} \ot 1) \cong J_{1,n}$, Lemma~\ref{lem:unique-masa-one-sided-dimension-drop} applies to its Cartan subalgebra $(e_{11} \ot 1) B_{[0,\veps]}$. We thus obtain a partial isometry $u_1 \in \cont_b((0,\veps],\rM_m \otimes \rM_n)$ whose support and range projection equal $e_{11} \ot 1$, which normalises $(J_{m, n})_{[0, \veps]}$ and satisfies 
\begin{equation*}
u_1 B_{[0,\veps]} u_1^* = (e_{11} \otimes 1) (C_{m,n})_{[0, \veps]}.
\end{equation*}
Put $u_i = (e_{i1} \ot 1) u_1 (e_{1i} \ot 1)$ and $u = u_1 + \dotsm + u_m \in \cU(\cont_b((0,\veps],\rM_m \otimes \rM_n))$. By construction, $\Ad u$ defines an automorphism of $(J_{m, n})_{[0, \veps]}$.  By making $\veps > 0$ smaller if necessary, we may assume by Lemma \ref{lem:partial-isometries-in-normaliser} that $e_{i1} \ot 1$ and $e_{1i} \ot 1$ are elements of $\rN_{(J_{m,n})_{[0,\veps]}}(B_{[0, \veps]})$. It therefore follows that $u B_{[0,\veps]} u^* = (C_{m,n})_{[0,\veps]}$. We can now extend $\Ad u$ to an automorphism of $J_{m,n}$ which is the identity on $[2\veps, 1]$.  Conjugating $B$ by this automorphism, we may assume that $B_{[0,\veps]} = (C_{m,n})_{[0,\veps]}$.

Let $0 < \delta < \veps$.  By Theorem \ref{thm:li-renault}, there is a unitary $u \in \cont([\delta, 1],\rM_m \otimes \rM_n)$ such that $u B_{[\delta, 1]}u^* = \cont([\delta, 1], \rD_m \ot \rD_n)$.  Then 
\begin{equation*}
  u|_{[\delta, \veps]}
  \in
  \cN_{\cont([\delta, \veps], \rM_m \ot \rM_n)}(\cont([\delta, \veps], \rD_m \ot \rD_n))
  =
  \cont([\delta, \veps],\cN_{\rM_m \ot \rM_n}(\rD_m \ot \rD_n))
  \eqstop
\end{equation*}
Hence, multiplying $u$ with some unitary from $\cont([\delta, 1], \cN_{\rM_m \ot \rM_n}(\rD_m \ot \rD_n))$, we may assume that $u|_{[\delta, \veps]} \equiv 1$.  Then $u$ extends to a unitary in $J_{m,n}$ that satisfies $u|_{[0,\veps]} \equiv 1$.  This unitary conjugates $B$ onto the standard Cartan subalgebra of $J_{m,n}$.
\end{proof}

%%% Local Variables:
%%% mode: latex
%%% TeX-master: "cartan-dimension-drop-algebra"
%%% End:

%% file: non-degenerate.tex
\section{Non-degenerate Cartan subalgebras}
\label{sec:non-degenerate-cartan}

The main purpose of this section is to prove Theorem \ref{thm:cartans-non-degenerate} characterising those stabilised dimension drop algebras, in which every Cartan subalgebra is non-degenerate.
% stating that every Cartan subalgebra of $I_{m,n}$ with $m$ and $n$ coprime is non-degenerate.  We also provide an example of a degenerate Cartan subalgebra, demonstrating necessity of our approach.
We start with a proposition characterising non-degenerate Cartan subalgebras as \Cstar-diagonals in the sense of Kumjian \cite{kumjian86}.
\begin{proposition}
  \label{prop:cstar-diagonals}
  Let $A \subset I_{m,n,o}$ be a Cartan subalgebra of a stabilised dimension drop algebra.  Then $A$ is non-degenerate if and only if it is a \Cstar-diagonal.
\end{proposition}
\begin{proof}
Pure states on $I_{m,n,o}$ are exactly of the form $\vphi \circ \operatorname{ev}_t$ for some $t \in [0,1]$ and some pure state $\vphi$ of $(I_{m,n,o})_t$. Similarly, pure states on $A$ are all of the form $\psi \circ \operatorname{ev}_t$ for some $t \in [0,1]$ and some pure state $\psi$ of $A_t$. The claim now follows as $A$ is a non-degenerate Cartan subalgebra of $I_{m,n,o}$ if and only if $A_t \subset (I_{m,n,o})_t$ has the unique extension property for all $t \in [0,1]$.
\end{proof}

\begin{lemma}
  \label{lem:completely-degenerate-cartan}
  Let $B \subset J_{m,n}$ be a Cartan subalgebra such that $B_0 = \CC$.  Then $m$ divides $n$.
\end{lemma}
\begin{proof}
Let $\rE: J_{m,n} \ra B$ be the unique conditional expectation onto $B$.  Using Theorem \ref{thm:li-renault}, one checks that $B_{(0,1]} \subset (J_{m,n})_{(0,1]} \cong \cont((0,1], \rM_m \ot \rM_n)$ is a Cartan subalgebra, so that $B_t \subset (J_{m,n})_t \cong \rM_m \ot \rM_n$ is a MASA for all $t \in (0,1]$. Let $(E_t)_{t \in (0,1]}$ be the fibres of $\rE$, that is, $\rE_t: \rM_m \ot \rM_n \ra \rM_m \ot \rM_n$ is defined by $E_t(f(t)) = E(f)(t)$ for all $f \in C_0((0,1],\rM_m \ot \rM_n)$. For each $t \in (0,1]$, $\rE_t: \rM_m \ot \rM_n \ra \rM_m \ot \rM_n$ is a conditional expectation onto some MASA.  Let $\rF: \rM_m \ot \rM_n \ra \rM_m \ot \rM_n$ be some limit point for this family, as $t$ approaches $0$ via some sequence $(t_k)_{k \in \NN}$. Note that such an $\rF$ always exists due to compactness of the unit ball of $\cB(\rM_m \ot \rM_n)$ and that $\rF$ again is a conditional expectation onto some MASA $C \subset \rM_m \ot \rM_n$.  Further, $F|_{M_m \ot 1} = \tr$ holds.

Since $B \subset J_{m,n}$ is a Cartan subalgebra, its normaliser spans $J_{m,n}$.  In particular $(N_{J_{m,n}}(B))_0$ spans $\rM_m \ot 1$.  But if $x \in N_{J_{m,n}}(B)$, then $x_0^*x_0, x_0x_0^* \in B_0 = \CC 1$, showing that the unitaries in $(N_{J_{m,n}}(B))_0$ already span $\rM_m \ot 1$. However, $(N_{J_{m,n}}(B))_0 \subset N_{\rM_m \ot \rM_n}(C) \cap \rM_m \ot 1$, showing that $\cN_{\rM_m \ot \rM_n}(C) \cap \rM_m \ot 1$ spans $\rM_m \ot 1$ as well.

Write $H = \CC^m \ot \CC^n$ and consider the natural action of $\rM_m \ot \rM_n$ on $H$.  Let $\{\xi_{ij}$ ~|~ $1\leq i \leq m,\  1 \leq j \leq n\}$ be an orthonormal basis of $H$ whose associated orthogonal projections $p_{ij}$ generate the MASA $C$. Define linear subspaces
  \begin{equation*}
  H_{ij} = \{(x \otimes 1)\xi_{ij} ~|~ x \in \rM_m\} \subset H
  \end{equation*}
  for $i\in \{1,\ldots,m\}$ and $j \in \{1,\ldots,n\}$.  Denote by $G = \{ u \in \cU(m) ~|~ u \ot 1 \in \cN_{\rM_m \ot \rM_n}(C)\}$. Observe that for each $u \in G$ and $i, j$ there are some $k,l$ such that $(u \ot 1)\xi_{ij} \in  \bT \xi_{kl}$.   Assume that there are $u_1, u_2 \in G$ such that $(u_1 \ot 1)\xi_{ij} \in \bT(u_2 \ot 1)\xi_{ij}$.  Then $u = u_1 u_2^*$ satisfies $(u \ot 1) \xi_{ij} \in \bT \xi_{ij}$.  It follows that $p_{ij} (u \ot 1) p_{ij} \in \bT p_{ij}$ and hence
  \begin{equation*}
    \tr(u \ot 1) = \rF(u \ot 1) = \sum_{k,l} p_{kl} (u \ot 1) p_{kl} \in \bT
    \eqstop
  \end{equation*}
  In conclusion, $u \in \bT$ meaning that $u_1, u_2$ are linearly dependent.  Since $\lspan G = \rM_m$, we can find a basis $(v_k)_k$ of $\rM_m$ inside $G$.  By the previous calculation, $((v_k \ot 1)\xi_{ij})_k$ is a basis of $H_{ij}$ showing that $\dim H_{ij} = m^2$.  Further, the sum of all $H_{ij}$ coincides with $H$.  Observe also that the equality 
\begin{equation*}
  H_{ij} = \lspan\{(u \otimes 1)\xi_{ij} ~|~ u \in G\} = \lspan\{\xi_{kl} ~|~ \exists u \in G \text{ such that } (u \otimes 1)\xi_{ij} = \xi_{kl}\}
  \end{equation*}
  implies that two subspaces $H_{ij}$ and $H_{kl}$ are either equal or orthogonal. This shows that $m^2$ divides $mn$ and hence $m$ divides $n$.
\end{proof}

We continue with the prototypical example of a degenerate Cartan subalgebra in $J_{m,m}$, which will be an important ingredient in the proof of Theorem \ref{thm:cartans-non-degenerate}.
\begin{example}
  \label{ex:completely-degenerate-cartan}
  We provide an example of a Cartan subalgebra $B \subset \rM_m \ot \rM_m$ such that the conditional expectation $\rE_B: \rM_m \ot \rM_m \ra B$ satisfies $\rE_B|_{\rM_m \ot 1} = \tr$ (i.e. $B$ and $\rM_m \ot 1$ are orthogonal in the sense of Popa \cite{popa83-orthogonal}) and such that $(\rM_m \ot 1) \cap N_{\rM_m \ot \rM_m}(B)$ generates $\rM_m \ot 1$.  Once this algebra is constructed, $\cont([0,1], B) \cap J_{m,m}$ provides an example of a degenerate Cartan subalgebra in $J_{m,m}$.  Let us explicitly check all properties of a Cartan subalgebra:
  \begin{itemize}
  \item The conditional expectation $\rE_B: \rM_m \ot \rM_m \ra B$ induces a conditional expectation $\id_{\cont([0,1])} \ot \rE_B: \cont([0,1], \rM_m \ot \rM_m) \ra \cont([0,1], B)$, which restricts to a faithful conditional expectation $\rE: J_{m,m} \ra \cont([0,1], B) \cap J_{m,m}$ thanks to the orthogonality assumption $\rE_B|_{\rM_m \ot 1} = \tr$.
  \item The algebra $\cont([0,1], B) \cap J_{m,m} \subset J_{m,m}$ is maximally abelian.  Indeed, if $f \in J_{m,m}$ commutes with $\cont([0,1], B) \cap J_{m,m} \subset J_{m,m}$, then $f(t) \in B$ for all $t \in (0,1]$.  By continuity, it follows that $f(0) \in B \cap (\rM_m \ot 1)$.
  \item In order to show that $\cont([0,1], B) \cap J_{m,m}$ is regular in $J_{m,m}$, we first observe that its restriction to $(0,1]$ is a standard Cartan subalgebra, which is hence regular.  It thus suffices to show that every element $x_0 \in (\rM_m \ot 1) \cap N_{\rM_m \ot \rM_m}(B)$ can be extended to some $x \in N_{J_{m,m}}(\cont([0,1], B) \cap J_{m,m})$.  But this is trivial: the constant function with value $x_0$ lies in $N_{J_{m,m}}(\cont([0,1], B)$.
  \end{itemize}

  \noindent Let us now construct $B$.  We write $H = \rM_m$ for the $m \times m$-matrices endowed with the scalar product
  \begin{equation*}
    \langle \hat T , \hat S \rangle
    =
    \tr(S^*T)
    \eqstop
  \end{equation*}
  Let $\pi:\rM_m \ra \bo(H)$ be the $*$-representation induced by left multiplication.  Observe that under the isomorphism between $H$ and $\CC^m \ot \CC^m$ sending $\widehat{e_{ij}}$ to $e_i \otimes e_j$, $\pi$ corresponds to the first tensor factor embedding into $\rM_m \ot \rM_m \cong \bo(H)$.
  
  Let $\lambda \in \bT$ be a primitive m-th root of unity and consider the element $U = \operatorname{diag}(1, \lambda, \lambda^2, \dotsc, \lambda^{m-1}) \in \rM_m$.  Further let $V$ be defined as the permutation matrix associated to the cycle $(1 \, 2 \, \dotsm \, m) \in \Sym(m)$.  Observe the commutation relation
  \begin{equation*}
    UV = \lambda VU
    \eqstop
  \end{equation*}
  Note also that elements of the form $\widehat{U^iV^j}$ for $i,j \in \{1, \dotsc, m\}$ form an orthonormal basis of $H$.  We consider the rank one projections $p_{ij}: H \ra \CC \widehat{U^iV^j}$ and let $B$ be the subalgebra of $\bo(H)$ generated by them.  Note that $B$ is an abelian subalgebra of dimension $m^2$ and hence a MASA. For $x \in \rM_m$ and $\hat{y} \in H$ it holds that
  \begin{equation*}
  	p_{ij}\pi(x)p_{ij}\hat{y}
  	=
  	\langle \widehat{xU^iV^j},\widehat{U^iV^j}\rangle \langle \hat{y},\widehat{U^iV^j}\rangle \widehat{U^iV^j}
  	= \tr(V^{-j}U^{-i} x U^i V^j) p_{ij}\hat{y}
  	= \tr(x)p_{ij}\hat{y}.
  \end{equation*}    
  Hence, the conditional expectation $\rE_B: \bo(H) \ra B$ restricted to $\pi(\rM_m)$ is calculated as
  \begin{equation*}
    \rE_B(\pi(x))
    =
    \sum_{i,j = 1}^m p_{ij} \pi(x) p_{ij}
    =
    \tr(x)
    \eqcomma
  \end{equation*}
  showing that $\rE_B|_{\pi(\rM_m)} = \tr$.  Next note that $\rM_m = \lspan \{U^i V^j \mid i, j \in \{1, \dotsc, m\}\}$ and the commutation relation $UV = \lambda VU$ yields that $\pi(U^iV^j) \in N_B(\bo(H))$ as
  \begin{equation*}
    \pi(U^iV^j) p_{kl} \pi(U^iV^j)^* \in \CC p_{(k+i), (l+j)}
    \eqcomma
  \end{equation*}
where the indices are considered mod $m$.
\end{example}

\begin{example}
  \label{ex:degenerate-cartan-non-coprime}
  Let $m,n \in \NN$ be non-coprime, say $d \neq 1$ divides $m$ and $n$.  By Example \ref{ex:completely-degenerate-cartan}, there is a Cartan subalgebra $B \subset J_{d, d}$ satisfying $\dim B_0 = 1$.  Writing $\rM_m \cong \rM_{\frac{m}{d}} \ot \rM_d$ and $\rM_n \cong \rM_d \ot \rM_{\frac{n}{d}}$, then the same argument as in the beginning of Example \ref{ex:completely-degenerate-cartan} shows that
  \begin{equation*}
    \{ f \in J_{m,n} \mid f(t) \in \rD_{\frac{m}{d}} \ot B_t \ot \rD_{\frac{n}{d}} \} \subset J_{m,n}
  \end{equation*}
is a degenerate Cartan subalgebra, since its dimension at $0$ equals $\frac{mn}{d^2}$.
\end{example}

\begin{theorem}
  \label{thm:cartans-non-degenerate}
  % Let $m,n \in \NN$ be coprime.  Then every Cartan subalgebra in $J_{m,n}$ is non-degenerate.  Consequently, every Cartan subalgebra in $I_{m,n}$ is non-degenerate.
  The following two statements are equivalent for a stabilised dimension drop algebra $I_{m,n,o}$.
  \begin{itemize}
  \item $m,n,o$ are pairwise coprime.
  \item Every Cartan subalgebra in $I_{m,n,o}$ is non-degenerate.
   \item Every Cartan subalgebra of $I_{m,n,o}$ is a \Cstar-diagonal, that is, it has the unique extension property.
  \end{itemize}
\end{theorem}
\begin{proof}
By Proposition \ref{prop:cstar-diagonals}, a Cartan subalgebra of $I_{m,n,o}$ is non-degenerate if and only it is a \Cstar-diaogonal.  It hence suffices to prove the equivalence between the first two statements of the theorem.

Assume that $I_{m,n,o}$ contains a degenerate Cartan subalgebra $B$. Possibly replacing $I_{m,n,o}$ by $I_{n,m,o}$, we may assume that $\dim B_0 \neq mo$.  Restricting to the interval $[0,\frac{1}{2}]$, we obtain a degenerate Cartan subalgebra of $J_{mo,n}$.  Since if $mo$ and $n$ are not coprime, then either $m$ and $n$ or $o$ and $n$ are not coprime, it thus suffices to consider one-sided dimension drop algebras in what follows.

  We let $m,n \in \NN$ be arbitrary and assume that $B \subset J_{m,n}$ is a degenerate Cartan subalgebra.  Our aim is to show that $m,n$ are not coprime.  Since $B$ is degenerate, there is a minimal projection $p_0 \in B_0$ with rank strictly bigger than one.  Let $p_0, \dotsc, p_K \in B_0$ be the minimal projections and $k \in \{1, \dotsc, K\}$.  Since the span of $N_{J_{m,n}}(B)$ is dense in $J_{m,n}$, it follows that there is a non-zero element $x \in p_0 N_{\rM_m \ot 1}(B_0) p_k$.  It follows that $x^*x \in p_kB_0 p_k = \CC p_k$ and $xx^* \in p_0 B_0 p_0 = \CC p_0$ are non-zero elements.  So $\operatorname{rank}(p_k)= \operatorname{rank}(p_0)$.  From
  \begin{equation*}
    m = \sum_{k = 0}^K \operatorname{rank}(p_k) = (K + 1)\operatorname{rank}(p_0)
  \end{equation*}
  it follows that $\operatorname{rank}(p_0)$ divides $m$.  Let us write $d = \operatorname{rank}(p_0)$.

  Extending $p_0$ to an element $p \in B$ such that $p|_{[0, \veps]}$ is a projection, we consider $(p B p)_{[0,\veps]} \subset (p J_{m,n} p)_{[0, \veps]}$, which is isomorphic with a Cartan subalgebra $C \subset J_{d,n}$ satisfying $C_0 = \CC1$.  So Lemma \ref{lem:completely-degenerate-cartan} applies to show that $d$ divides $n$.  Since $d \neq 1$, this shows that $m$ and $n$ have a non-trivial common divisor.

  Let us now assume that $m,n,o$ are not pairwise coprime.  Possibly changing the role of $m$ and $n$, we may assume without loss of generality that there is some $d \neq 1$ that divides $mo$ and $n$.  Example \ref{ex:degenerate-cartan-non-coprime} provides us with a degenerate Cartan subalgebra
  \begin{equation*}
    B
    \subset
    \{f \in \cont([0,\frac{1}{2}], \rM_m \ot \rM_n \ot \rM_o) \mid f(0) \in \rM_m \ot 1 \ot \rM_o\}
    \cong
    J_{mo,n}.
  \end{equation*}
  If $u \in \cont([\frac{1}{2}, 1], \rM_m \ot \rM_n \ot \rM_o )$ denotes a unitary such that
  \begin{equation*}
    u_{\frac{1}{2}} (\rD_m \ot \rD_n \ot \rD_o) u_{\frac{1}{2}}^* = B_{\frac{1}{2}}
  \end{equation*}
  and $u_1 = 1$, then
  \begin{equation*}
    \Bigl \{
    f \in \cont([0,1], \rM_m \ot \rM_n \ot \rM_o) \mid
    f(t) \in
    \begin{cases}
      B_t  \eqcomma & \text{if }t \in [0,\frac{1}{2}] \eqcomma \\
      u_t (\rD_m \ot \rD_n \ot \rD_o)u_t^* \eqcomma & \text{if } t \in [\frac{1}{2}, 1) \eqcomma\\
      1 \ot \rM_n \ot \rM_o \eqcomma & \text{if } t = 1 \eqstop
    \end{cases}
    \Bigr \}
  \end{equation*}
  is a degenerate Cartan subalgebra in $I_{m,n,o}$.  
\end{proof}

%%% Local Variables:
%%% mode: latex
%%% TeX-master: "cartan-dimension-drop-algebra"
%%% End:

%% file: twisted.tex
\section{Twisted standard Cartan subalgebras}
\label{sec:twisted}

In this section, we will parametrise non-degenerate Cartan subalgebras of arbitrary stabilised dimension drop algebras. It turns out that all non-degenerate Cartan subalgebras in stabilised dimension drop algebras $I_{m,n,o}$ are twisted versions of the following standard Cartan subalgebra inherited from the inclusion into $\cont([0,1], \rM_m \ot \rM_n \ot \rM_o)$.
\begin{proposition}
  \label{prop:standard-cartan-properties}
  The \Cstar-subalgebra $B := \cont([0,1], \rD_m \ot \rD_n \ot \rD_o) \cap I_{m,n,o} \subset I_{m,n,o}$ is a Cartan subalgebra.  Moreover, $B$ has no non-trivial projections, or equivalently its spectrum is connected.
\end{proposition}
\begin{proof}
  We have to check that $B \subset I_{m,n,o}$ admits a faithful conditional expectation, is a MASA and that it is regular.  First note that the conditional expectation $\cont([0,1], \rM_m \ot \rM_n \ot \rM_o) \ra \cont([0,1], \rD_m \ot \rD_n \ot \rD_o)$ restricts to a faithful conditional expectation $I_{m,n,o} \ra B$.  Further, $B$ is a MASA in $I_{m,n,o}$, since $B_t \subset (I_{m,n,o})_t$ is a MASA for all $t \in [0,1]$.  It remains to show that $B$ is regular in $I_{m,n,o}$.  A straightforward partition of unity argument shows that to this end, it suffices to show that for each $t \in [0,1]$ there is a neighbourhood $V$ of $t$ such that $B|_V \subset (I_{m,n,o})|_V$ is regular.  For $t \in (0,1)$ we can find a neighbourhood $V$ of $t$ such that $B|_V \subset (I_{m,n,o})|_V$ is isomorphic with $\cont([0,1], \rD_m \ot \rD_n \ot \rD_o) \subset \cont([0,1], \rM_m \ot \rM_n \ot \rM_o)$.  So it remains to treat the points $0$ and $1$.  For each of them there is a neighbourhood $V$ such that $B|_V \subset (I_{m,n,o})|_V$ is the standard Cartan subalgebra of a one-sided dimension drop algebra, which is regular by Example \ref{ex:cartan-one-sided-dimension-drop}.

The spectrum of $B$ is homeomorphic with the quotient of $\ul m \times \ul n \times \ul o \times [0,1]$ by the relations $(i, j, k, 0) \sim (i,j', k, 0)$ and $(i,j,k,1) \sim (i', j,k, 1)$ for all $i,i' \in \ul m$, all $j,j' \in \ul n$, and all $k \in \ul o$.  It is hence connected.  Equivalently, $B$ does not have any non-trivial projections.
\end{proof}

\begin{definition}
  \label{def:standard-cartan}
  Let $I_{m,n,o}$ be a stabilised dimension drop algebra.  We call $\cont([0,1], \rD_m \ot \rD_n \ot \rD_o) \cap I_{m,n,o}$ the standard Cartan subalgebra of $I_{m,n,o}$.
\end{definition}

Proposition \ref{prop:standard-cartan-properties} allows us to exhibit a (non-degenerate) Cartan subalgebra of a dimension drop algebra, which is not conjugate to the standard Cartan subalgebra, simply because it contains non-trivial projections.

\begin{example}
  \label{ex:non-unique-Cartan}
  Let $I_{m,n,o}$ be a stabilised dimension drop algebra with $m$ and $n$ not coprime.  Then $I_{m,n,o}$ admits a non-degenerate Cartan subalgebra with a non-trivial projection and hence $I_{m,n,o}$ does not have a unique non-degenerate Cartan subalgebra up to conjugacy by an automorphism.  It suffices to consider the case $m = n$ and $o = 1$.  Since the tensor flip on $\rM_m \ot \rM_m$ is inner, we may find a unitary $u \in \cont([0,1], \rM_m \ot \rM_m)$ such that $u_0 = 1$ and $u_1$ implements the tensor flip.  Then $B = \{f \in I_{m,m} \mid f(t) \in u_t (\rD_m \ot \rD_m) u_t^* \}$ is a Cartan subalgebra of $I_{m,m}$ - this is proven verbatim as is Proposition \ref{prop:standard-cartan-properties}.  Now $u (e_{11} \ot 1) u^* \in B$ is a non-trivial projection.
\end{example}

We are now going to expand the idea of Example \ref{ex:non-unique-Cartan} in order to obtain a finite family of Cartan subalgebras of $I_{m,n,o}$ parametrised by elements of  $\Sym(\ul m \times \ul n \times \ul o)$.  In Theorem \ref{thm:non-conjugate-cartans}, we will classify these up to conjugacy by an automorphism and in Theorem \ref{thm:classification-cartan-dimension-drop} we will show that every non-degenerate Cartan subalgebra of $I_{m,n,o}$ is conjugate to one of the following kinds.
\begin{example}
  \label{ex:twisted-standard-Cartan}
  Let $\sigma \in \Sym(\ul m \times \ul n \times \ul o)$.  We identify $\sigma$ with its permutation matrix and hence a unitary in $\cU(m) \times \cU(n) \times \cU(o)$.  Let $u \in \cont([\frac{1}{3}, \frac{2}{3}], \rM_m \ot \rM_n \ot \rM_o)$ be a unitary such that $u_{\frac{1}{3}} = \sigma$ and $u_{\frac{2}{3}} = 1$.  We set
  \begin{multline*}
    B_{\sigma, u}
    =
    \{
    f \in I_{m,n,o} \mid
    f(t) \in \rD_m \ot \rD_n \ot \rD_o \text{ for } t \in [0,\frac{1}{3}] \cup [\frac{2}{3}, 1] \\
    \text{ and } f(t) \in u_t (\rD_m \ot \rD_n \ot \rD_o) u_t^* \text{ for } t \in [\frac{1}{3}, \frac{2}{3}]
    \}
    \eqstop
  \end{multline*}
  If $v \in \cont([\frac{1}{3}, \frac{2}{3}], \rM_m \ot \rM_n \ot \rM_o)$ is another unitary satisfying $v_{\frac{1}{3}} = \sigma$ and $v_{\frac{2}{3}} = 1$, then $uv^*$ extends to a unitary $w \in I_{m,n,o}$ such that $w|_{[0,\frac{1}{3}] \cup [\frac{2}{3}, 1]} = 1$.  Hence, $B_{\sigma, u} = w B_{\sigma, v} w^*$.  We will henceforth suppress the choice of $u$ and write $B_\sigma$.  Note that $\rD_m \ot \rD_n \ot \rD_o$ is normalised by $\sigma$, so that $B_{\frac{1}{3}} = B_{\frac{2}{3}} = \rD_m \ot \rD_n \ot \rD_o$.  The proof of Proposition \ref{prop:standard-cartan-properties} therefore applies verbatim to show that $B_\sigma$ is a Cartan subalgebra in $I_{m,n,o}$. 
 
    Next, we argue that the choice of $\frac{1}{3}$ and $\frac{2}{3}$ is arbitrary.  For every $0 < t_1 < t_2 < 1$, there is an orientation preserving homeomorphism of $[0,1]$ satisfying $\frac{1}{3} \mapsto t_1$ and $\frac{2}{3} \mapsto t_2$.   It extends to an automorphism $\alpha$ of $I_{m,n,o}$, since it must fix the endpoints of $[0,1]$.  Replacing $\frac{1}{3}$ by $t_1$ and $\frac{2}{3}$ by $t_2$, we can apply the same construction as above to obtain a Cartan subalgebra which is conjugate to $B_\sigma$ by the automorphism $\alpha$.
\end{example}

\begin{definition}
  \label{def:twisted-standard-cartan}
   The Cartan subalgebras of stabilised dimension drop algebras constructed in Example \ref{ex:twisted-standard-Cartan} are called twisted standard Cartan subalgebras of $I_{m,n,o}$.
\end{definition}

Before in Theorem \ref{thm:non-conjugate-cartans} we can classify the Cartan subalgebras constructed in Example \ref{ex:twisted-standard-Cartan}, we need two lemmas for its proof.
\begin{lemma}
  \label{lem:homeomorphism-realised}
  Let $B_\sigma \subset I_{m,n,o}$ be a twisted standard Cartan subalgebra of a stabilised dimension drop algebra.  Let $\vphi \in \operatorname{Homeo}^+([0,1])$ be an orientation preserving homeomorphism of the interval.  Then there is an automorphism $\alpha \in \Aut(B_\sigma \subset I_{m,n,o})$ such that $\alpha|_{\cont([0,1])} = \vphi^*$, that is, $\alpha$ is given by pre-composition with $\vphi$.
\end{lemma}
\begin{proof}
  Let $B_\sigma$ be constructed from the unitary $u \in \cont([\frac{1}{3}, \frac{2}{3}], \rM_m \ot \rM_n \ot \rM_o)$ as described in Example~\ref{ex:twisted-standard-Cartan}.  Let $\veps > 0$ be chosen such that $\veps \leq \min \{ \frac{1}{3}, \vphi^{-1}(\frac{1}{3})\}$ and $\max \{\frac{2}{3}, \vphi^{-1}(\frac{2}{3})\} \leq 1 - \veps$.  We extend $u$ to a unitary $u \in \cont([\veps, 1 - \veps], \rM_m \ot \rM_n \ot \rM_o)$ by putting $u_{[\veps, \frac{1}{3}]} \equiv \sigma$ and $u_{[\frac{2}{3}, 1 - \veps]} \equiv 1$.  Define the unitary ${v \in \cont([\veps, 1 - \veps], \rM_m \ot \rM_n \ot \rM_o)}$ by $v_t = u_{\vphi(t)}$ for all $t \in \vphi^{-1}([\frac{1}{3}, \frac{2}{3}])$ and $v|_{[\veps, \vphi^{-1}(\frac{1}{3})]} \equiv \sigma$ and $v|_{[\vphi^{-1}(\frac{2}{3}), 1 - \veps]} \equiv 1$.  Then we have $(u v^*)_\veps = (u v^*)_{1 - \veps} = 1$, so that we can extend $uv^* $ by $1$ to a unitary in $\cont([0,1], \rM_m \ot \rM_n \ot \rM_o)$.  Put $\alpha = \Ad uv^* \circ \vphi^* \in \Aut(I_{m,n,o})$ and note that $\alpha|_{\cont([0,1])} = \vphi^*$.  We check that $\alpha(B_\sigma) = B_\sigma$.  Let $f \in B_\sigma$.  Then
  \begin{equation*}
    \alpha(f) (t)
    =
    \Ad (uv^*)_t (\vphi^* f)(t)
    =
    \Ad (uv^*)_t f( \vphi(t))
    \eqstop
  \end{equation*}
  Since for $t \in \vphi^{-1}([\veps, 1 - \veps])$ it holds that $f(\vphi(t)) \in u_{\vphi(t)} (\rD_m \ot \rD_n \ot \rD_o) u_{\vphi(t)}^*$, we find that 
  \begin{align*}
    \Ad (uv^*)_t f( \vphi(t))    
    & \in
      u_t v_t^* u_{\vphi(t)} (\rD_m \ot \rD_n \ot \rD_o) u_{\vphi(t)}^* v_t u_t^* \\
    & =
      u_t u_{\vphi(t)}^* u_{\vphi(t)} (\rD_m \ot \rD_n \ot \rD_o) u_{\vphi(t)}^* u_{\vphi(t)} u_t^* \\
    & =
      u_t (\rD_m \ot \rD_n \ot \rD_o) u_t^* 
    \eqcomma
  \end{align*}
  and $\alpha(f) (t) \in \rD_m \ot \rD_n \ot \rD_o$ otherwise.  This finishes the proof of the lemma.
\end{proof}

\begin{notation}
  \label{not:orientation-preserving-automorphism}
  Note that the centre of $I_{m,n,o}$ is isomorphic with $\cont([0,1])$.  Hence, every automorphism $\alpha$ of $I_{m,n,o}$ induces a homeomorphism of $[0,1]$.  We call $\alpha$ orientation preserving, if this homeomorphism of $[0,1]$ is orientation preserving.
\end{notation}

\begin{lemma}
  \label{lem:flip-twisted-standard-cartan}
  Let $\Phi: [0,1] \ra [0,1]$ be the flip defined by $\Phi(t) = 1 - t$, let $\nu \in \Sym(\ul m \times \ul m \times \ul o)$ be the flip on $\ul m \times \ul m \times \ul o$ satisfying $\nu(i,j,k) = (j,i,k)$ and let $v \in \cont([0,1], \rM_m \ot \rM_m \ot \rM_o)$ be the unitary that is constant as a function with value $\nu$.  Then $\alpha = \Ad v \circ \Phi^* \in \Aut(I_{m,m,o})$ and $\Phi^*(B_{\sigma}) = B_{\nu \sigma^{-1} \nu}$ for all $\sigma \in \Sym(\ul m \times \ul m \times \ul o)$.
\end{lemma}
\begin{proof}
  We first show that $\Ad v \circ \Phi^* \in \Aut(\cont([0,1], \rM_m \ot \rM_m \ot \rM_o))$ normalises $I_{m,m,o}$ and hence restricts to the order two automorphism $\alpha \in \Aut(I_{m,m,o})$.  For $f \in I_{m,m,o}$ we have
  \begin{equation*}
    (\Ad v \circ \Phi^* f)(t)
    =
    \nu f(1 - t) \nu
    \in
    \begin{cases}
      \nu (1 \ot \rM_m \ot \rM_o) \nu = \rM_m \ot 1 \ot \rM_o \eqcomma & \text{ if } t = 0 \eqcomma \\
      \nu (\rM_m \ot 1 \ot \rM_o) \nu = 1 \ot \rM_m \ot \rM_o \eqcomma & \text{ if } t = 1
      \eqstop
    \end{cases}
  \end{equation*}
  So indeed, $\Ad v \circ \Phi^* f \in I_{m,m,o}$.

  If $u$ is a unitary defining $B_\sigma$ as in Example \ref{ex:twisted-standard-Cartan} and $f \in B_\sigma$, then $\alpha(f)(t) = \nu f(1 - t) \nu$, so that we obtain
  \begin{equation*}
    \alpha(f)(t)
    \in
    \nu (\rD_m \ot \rD_m \ot \rD_o) \nu
    = \rD_m \ot \rD_m \ot \rD_o
  \end{equation*}
  if $t \in [0,\frac{1}{3}) \cup (\frac{2}{3}, 1]$, and 
  \begin{equation*}
    \alpha(f)(t)
    \in
    \nu u_{1- t} (\rD_m \ot \rD_m \ot \rD_o) u_{1-t}^* \nu
    =
    \nu u_{1- t} \sigma^{-1} \nu (\rD_m \ot \rD_m \ot \rD_o) \nu \sigma u_{1-t}^* \nu
    \eqcomma
  \end{equation*}
  if $t \in [\frac{1}{3}, \frac{2}{3}]$.  Since $\nu u_{1 - \frac{1}{3}} \sigma^{-1} \nu = \nu \sigma^{-1} \nu$ and $\nu u_{1 - \frac{2}{3}} \sigma^{-1} \nu = \nu \sigma \sigma^{-1} \nu = 1$, we see that $\alpha(B_\sigma) \subset B_{\nu \sigma^{-1} \nu}$. Since $\alpha$ is of order two, this implies $\alpha(B_\sigma) = B_{\nu \sigma^{-1} \nu}$. This concludes the proof.
\end{proof}

Before we proceed to prove our next theorem, let us fix some notation for wreath products.
\begin{notation}
  \label{not:wreath-product}
  If $F$ is some group, then $F \wreath \Sym(n) = \bigoplus_{i = 1}^n F \rtimes \Sym(n)$ denotes the wreath product, where $\Sym(n)$ acts by permuting the copies of $F$.   We also use the notation $\Sym(n) \lwreath F = \Sym(n) \ltimes \bigoplus_{i = 1}^n F$ for notational convenience.

  Fixing $m,n,o \in \NN_{\geq 1}$, we obtain an embedding $\Sym(\ul m \times \ul o) \lwreath \Sym(\ul n) \leq \Sym(\ul m \times \ul n \times \ul o)$.  Concretely, an element of $\Sym(\ul m \times \ul o) \lwreath \Sym(\ul n)$ can be described as a pair $(\sigma_1, \sigma_2)$ with $\sigma_1 \in \Sym(\ul m \times \ul o)$ acting on the first and last coordinate of $\ul m \times \ul n \times \ul o$ and $\sigma_2: \ul m \times \ul n \times \ul o \ra \ul n$ is a map such that $\sigma_2(i, \, \cdot \, , k) \in \Sym(\ul n)$ for all $i \in \ul m$, $k \in \ul o$.  This way, $\sigma_1$ acts on the second coordinate of $\ul m \times \ul n \times \ul o$.  Similarly, we obtain an embedding $\Sym(\ul m) \wreath \Sym(\ul n \times \ul o) \leq \Sym(\ul m \times \ul n \times \ul o)$.

  We remark that considering permutations as unitaries, we have the equalities
  \begin{equation*}
    \Sym(\ul m \times \ul o) \lwreath \Sym(\ul n)
    =
    \Sym(\ul m \times \ul n \times \ul o) \cap \cN_{\rM_m \ot \rM_n \ot \rM_n}(\rD_m \ot 1 \ot \rD_o)
  \end{equation*}
  and similarly 
  \begin{equation*}
    \Sym(\ul m) \wreath \Sym(\ul n \times \ul o)
    =
    \Sym(\ul m \times \ul n \times \ul o) \cap \cN_{\rM_m \ot \rM_n \ot \rM_n}(1 \ot \rD_n \ot \rD_o)
    \eqstop
  \end{equation*}
\end{notation}

\begin{theorem}
  \label{thm:non-conjugate-cartans}
  Let $I_{m,n,o}$ be a stabilised dimension drop algebra and $\sigma, \pi \in \Sym(\ul m \times \ul n \times \ul o)$.
  \begin{itemize}
  \item $B_\sigma$ and $B_\pi$ are conjugate by an orientation preserving automorphism of $I_{m,n,o}$ if and only if they define the same element in $\Sym(\ul m \times \ul o) \lwreath \Sym(\ul n) \backslash \Sym(\ul m \times \ul n \times \ul o) / \Sym(\ul m) \wreath \Sym(\ul n \times \ul o)$.
  \item If $m \neq n$, then twisted standard Cartan subalgebras of $I_{m,n,o}$ are classified up to conjugacy by elements of $\Sym(\ul m \times \ul o) \lwreath \Sym(\ul n) \backslash \Sym(\ul m \times \ul n \times \ul o) / \Sym(\ul m) \wreath \Sym(\ul n \times \ul o)$
  \item If $m = n$ and $\nu \in \Sym(\ul m \times \ul m \times \ul o)$ denotes the flip of the first two coordinates, then $B_\sigma$ and $B_\pi$ are conjugate if and only if $\sigma \in (\Sym(\ul m \times \ul o) \lwreath \Sym(\ul n)) \cdot \{\pi, \nu \pi^{-1} \nu\} \cdot (\Sym(\ul m) \wreath \Sym(\ul n \times \ul o))$.
  \end{itemize}
\end{theorem}
\begin{proof}
  We first show that $B_\sigma$ depends up to conjugacy by an orientation preserving automorphism only on the class of $\sigma$ in $\Sym(\ul m \times \ul o) \lwreath \Sym(\ul n) \backslash \Sym(\ul m \times \ul n \times \ul o) / \Sym(\ul m) \wreath \Sym(\ul n \times \ul o)$.  Indeed, if $\rho_1 \in \Sym(\ul m \times \ul o) \lwreath \Sym(\ul n)$ and $\rho_2 \in \Sym(\ul m) \wreath \Sym(\ul n \times \ul o)$, we find a unitary $v \in \cont([0,1], \rM_m \ot \rM_n \ot \rM_o)$ such that $v|_{[0, \frac{1}{3}]} = \rho_1$ and $v|_{[\frac{2}{3}, 1]} = \rho_2$.   It normalises $I_{m,n,o}$ and hence defines an automorphism $\alpha \in \Aut(I_{m,n,o})$.  If now $B_\sigma$ is constructed from a unitary $u \in \cont([\frac{1}{3}, \frac{2}{3}], \rM_m \ot \rM_n \ot \rM_o)$, then we obtain for $f \in B_\sigma$
\begin{equation*}
  v f(t) v^*
  \in
  \begin{cases}
    \rho_1 (\rD_m \ot 1 \ot \rD_o) \rho_1^{-1}  = \rD_m \ot 1 \ot \rD_o \eqcomma & \text{ if } t = 0 \eqcomma  \\
    \rho_1 (\rD_m \ot \rD_n \ot \rD_o) \rho_1^{-1}  = \rD_m \ot \rD_n \ot \rD_o \eqcomma & \text{ if } t \in (0,\frac{1}{3}] \eqcomma \\
    v_t u_t (\rD_m \ot \rD_n \ot \rD_o) u_t^* v_t^* \eqcomma & \text{ if } t \in [\frac{1}{3}, \frac{2}{3}] \eqcomma \\
    \rho_2 (\rD_m \ot \rD_n \ot \rD_o) \rho_2^{-1}  = \rD_m \ot \rD_n \ot \rD_o \eqcomma & \text{ if } t \in [\frac{2}{3},1) \eqcomma \\
    \rho_2 (1 \ot \rD_n \ot \rD_o) \rho_2^{-1}  = 1 \ot \rD_n \ot \rD_o \eqcomma & \text{ if } t = 1
    \eqstop
  \end{cases}
\end{equation*}
Since $v_t u_t (\rD_m \ot \rD_n \ot \rD_o) u_t^* v_t^* = v_t u_t \rho_2^{-1} (\rD_m \ot \rD_n \ot \rD_o) \rho_2 u_t^* v_t^*$ and $v_{1/3}u_{1/3} \rho_2^{-1} = \rho_1 \sigma \rho_2^{-1}$ as well as $v_{2/3} u_{2/3} \rho_2^{-1} = \rho_2 1 \rho_2^{-1} = 1$, we conclude that $\alpha(B_\sigma) = B_{\rho_1 \sigma \rho_2^{-1}}$.

Let now $\alpha \in \Aut(I_{m,n,o})$ be an orientation preserving automorphism in the sense of Notation \ref{not:orientation-preserving-automorphism}.  Assume that $\alpha$ satisfies $\alpha(B_\sigma) = B_\pi$.  We can apply Lemma \ref{lem:homeomorphism-realised} and assume that $\alpha|_{\cont([0,1])} = \id$.  Let us consider the elements $u_\sigma, u_\pi \in \cont([\frac{1}{3}, \frac{2}{3}], \rM_m \ot \rM_n \ot \rM_o)$ as described in Example \ref{ex:twisted-standard-Cartan}.  The automorphism $(\Ad u_\pi)^{-1} \circ \alpha \circ \Ad u_{\sigma}$ normalises the standard MASA $\cont([1/3, 2/3], \rD_m \ot \rD_n \ot \rD_o)$ and it is hence a continuous function with values in $((\bT^m \times \bT^n \times \bT^o) / \bT )\cdot  \Sym(\ul m \times \ul n \times \ul o)$.  In particular, its permutational part is constant.  Let us denote it by $[\alpha_t]$, $t \in [\frac{1}{3}, \frac{2}{3}]$.  Evaluating at $\frac{1}{3}$ and $\frac{2}{3}$, we obtain the equality 
\begin{equation*}
[(\Ad u_\pi)^{-1}_{1/3} \circ \alpha_{1/3} \circ (\Ad u_\sigma)_{1/3}] = [(\Ad u_\pi)^{-1}_{2/3} \circ \alpha_{2/3} \circ (\Ad u_\sigma)_{2/3}].
\end{equation*}
The definitions of $u_\sigma$ and $u_\pi$ give then $\pi^{-1} \circ [\alpha_{1/3}] \circ \sigma = [\alpha_{2/3}]$.  Put differently, we have
\begin{equation}
  \label{eq:simga-pi}
  [\alpha_{1/3}] \circ \sigma \circ [\alpha_{2/3}]^{-1}= \pi
  \eqstop
\end{equation}
Since $\alpha|_{(0,1/3]}$ normalises the diagonal MASA $\cont((0,1/3], \rD_m \ot \rD_n \ot \rD_o)$, it takes values in $((\bT^m \times \bT^n \times \bT^o)/\bT) \cdot \Sym(\ul m \times \ul n \times \ul o)$ and thus its permutational part is constant, say $[\alpha_{(0, 1/3]}] \equiv \rho_1$.  Now $\alpha_{(0,1/3]}$ is the restriction of an automorphism of $I_{m,n,o}$, so that $\rho_1$ must normalise $\rD_m \ot 1 \ot \rD_o$.  This implies $\rho_1 \in \Sym(\ul m \times \ul o) \lwreath \Sym(\ul n)$.  Similarly, we obtain $[\alpha_{[2/3, 1)}] \equiv \rho_2 \in \Sym(\ul m) \wreath \Sym(\ul n \times \ul o)$.  Combining this with (\ref{eq:simga-pi}), we obtain
\begin{equation*}
  \pi
  \in
  \bigl ( \Sym(\ul m \times \ul o) \lwreath \Sym(\ul n) \bigr )
  \sigma
  \bigl ( \Sym(\ul m) \wreath \Sym(\ul n \times \ul o) \bigr )
  \eqstop
\end{equation*}
This proves the first statement of the theorem.

Assume now $m \neq n$.  Then ${\rM_m \ot 1 \ot \rM_o} \not \cong {1 \ot \rM_n \ot \rM_o}$ implies that every automorphism of $I_{m,n,o}$ is orientation preserving.  This proves the second statement of the theorem.

Finally assume that $m = n$.  Lemma \ref{lem:flip-twisted-standard-cartan} shows that $B_\sigma$ is conjugate by a non-orientation preserving automorphism of $I_{m,m,o}$ to $B_{\nu \sigma^{-1} \nu}$.  Since the composition of two non-orientation preserving automorphism of $I_{m,m,o}$ is orientation preserving and $\nu (\Sym(\ul m \times \ul o)  \lwreath \Sym(\ul m)) \nu = \Sym(\ul m) \wreath \Sym(\ul m \times \ul o)$, this proves the last statement of the theorem.
\end{proof}

We now address the complete classification of Cartan subalgebras of stabilised dimension drop algebras.  It turns out that the twisted standard Cartan subalgebras exhaust all examples of non-degenerate Cartan subalgebras in there.
\begin{theorem}
  \label{thm:classification-cartan-dimension-drop}
  Every non-degenerate Cartan subalgebra of a stabilised dimension drop algebra is conjugate to a twisted standard Cartan subalgebra.
\end{theorem}
\begin{proof}
  Let $B \subset I_{m,n,o}$ be a non-degenerate Cartan subalgebra.  By Proposition \ref{prop:unique-cartan-one-sided-dimension-drop}, there are automorphisms $\alpha \in \Aut((I_{m,n,o})_{[0,\frac{1}{3}]})$ and $\beta \in \Aut((I_{m,n,o})_{[\frac{2}{3}, 1]})$ such that
  \begin{equation*}
    \alpha( B_{[0,\frac{1}{3}]} )
    =
    \{f \in \cont([0,\frac{1}{3}], \rD_m \ot \rD_n \ot \rD_o \mid f(0) \in \rD_m \ot 1 \ot \rD_o\}
  \end{equation*}
  and
  \begin{equation*}
    \beta (B_{[\frac{2}{3},1]})
    =
    \{f \in \cont([\frac{2}{3},1], \rD_m \ot \rD_n \ot \rD_o \mid f(1) \in 1 \ot \rD_n \ot \rD_o\}.
  \end{equation*}
  Extending $\alpha$ and $\beta$ to automorphisms of $I_{m,n,o}$ that satisfy $\alpha|_{[\frac{2}{3}, 1]} = \id$ and $\beta|_{[0, \frac{1}{3}]} = \id$, we may assume that the restrictions to $[0,\frac{1}{3}]$ and $[\frac{2}{3}, 1]$ of $B$ and the standard Cartan subalgebra of $I_{m,n,o}$ coincide.

  By Theorem \ref{thm:li-renault}, there is a unitary $u \in (I_{m,n,o})_{[\frac{1}{3}, \frac{2}{3}]}$ such that
  \begin{equation*}
    u B_{[\frac{1}{3}, \frac{2}{3}]} u^* = \cont([\frac{1}{3}, \frac{2}{3}], \rD_m \ot \rD_n \ot \rD_o)
    \eqstop
  \end{equation*}
  Then
  \begin{equation*}
    u_{\frac{2}{3}}
    \in
    \cN_{\rM_m \ot \rM_n \ot \rM_o} (\rD_m \ot \rD_n \ot \rD_o)
    =
    (\bT^m \times \bT^n \times \bT^o) \rtimes \Sym(\ul m \times \ul n \times \ul o)
    \eqstop
  \end{equation*}
  Multiplying $u$ with an element from $\cont([\frac{1}{3}, \frac{2}{3}], \bT^m \times \bT^n \times \bT^o)  \rtimes \Sym(\ul m \times \ul n \times \ul o)$, we may hence assume that $u_{\frac{2}{3}} = 1$.  Similarly, we obtain
  \begin{equation*}
    u_{\frac{1}{3}}
    \in
    (\bT^m \times \bT^n \times \bT^o) \rtimes \Sym(\ul m \times \ul n \times \ul o)
    \eqstop
  \end{equation*}
  Since $\bT^m \times \bT^n \times \bT^o$ is connected, we may multiply $u$ with a function from $\cont([\frac{1}{3}, \frac{2}{3}], \bT^m \times \bT^n \times \bT^o)$ to assume that in addition $u_{\frac{1}{3}} = \pi \in \Sym(\ul m \times \ul n \times \ul o)$.  It follows that $f \in B$ satisfies
  \begin{equation*}
    f(t) \in
    \begin{cases}
      \rD_m \ot \rD_n \ot \rD_o, & t \in [0,\frac{1}{3}] \cup [\frac{2}{3}, 1], \\
      u_t^* (\rD_m \ot \rD_n \ot \rD_o) u_t, & t \in [\frac{1}{3},\frac{2}{3}]
      \eqstop
    \end{cases}
  \end{equation*}
  It follows that $B$ is the twisted standard Cartan subalgebra associated with $\pi^{-1}$.  This completes the proof of the theorem.
\end{proof}

%%% Local Variables:
%%% mode: latex
%%% TeX-master: "cartan-dimension-drop-algebra"
%%% End:

%% file: counting.tex
\section{Counting Cartan subalgebras}
\label{sec:counting}

Let us start this section by counting conjugacy classes of non-degenerate Cartan subalgebras in the example of $I_{2,2}$, which is the easiest non-trivial case of all dimension drop algebras.
\begin{example}
  \label{ex:cartans-two-two}
  $I_{2,2}$ has at least 2 non-degenerate Cartan subalgebras up to conjugacy by an automorphism.  These are the standard Cartan subalgebra and the Cartan subalgebra constructed in Example \ref{ex:non-unique-Cartan}.  Indeed, both are not conjugate to each other, since their spectra are not homeomorphic.  Theorem \ref{thm:classification-cartan-dimension-drop} implies that these are all non-degenerate Cartan subalgebras of $I_{2,2}$ up to conjugacy. For this, we have to prove that
  \begin{equation*}
    \Sym(\ul 2) \lwreath \Sym(\ul 2) \bs \Sym(\ul 2 \times \ul 2) / \Sym(\ul 2) \wreath \Sym(\ul 2)
  \end{equation*}
  has cardinality 2.  This follows from a counting argument.  The group
  \begin{equation*}
    \Sym(\ul 2) \wreath \Sym(\ul 2) = (\Sym(\ul 2) \oplus \Sym(\ul 2)) \rtimes \Sym(\ul 2)
  \end{equation*}
  has cardinality $2 \cdot 2 \cdot 2 = 8$, while $\Sym(\ul 2 \times \ul 2)$ has cardinality $4 \cdot 3 \cdot 2 = 24$.  It follows that the quotient $\Sym(\ul 2 \times \ul 2) / \Sym(\ul 2) \wreath \Sym(\ul 2)$ has cardinality $3$.  Since $\Sym(\ul 2) \lwreath \Sym(\ul 2)$ acts non-trivially on this set, we infer that the orbit space has cardinality at most 2.  Combined with the initial observation that $I_{2,2}$ has at least 2 non-conjugate Cartan subalgebras, this shows that $I_{2,2}$ has exactly 2 non-conjugate non-degenerate Cartan subalgebras.
\end{example}

\subsection{Matrix combinatorics and Cartan subalgebras}
\label{sec:matrix-combinatorics}

In order to count conjugacy classes of non-degenerate Cartan subalgebras in any other - more complicated - case than the one described in Example \ref{ex:cartans-two-two}, it is necessary to improve on the inexplicit parametrisation provided by double cosets of permutations.  It turns out to be advantageous to replace permutations by their permutation matrices and then compress their information.  Indeed, Theorem~\ref{thm:cartan-classified-by-matrix} will describe conjugacy classes of non-degenerate Cartan subalgebras of dimension drop algebras in terms of certain congruence classes of matrices resulting from the following procedure.  This in turn will lead to the desired parametrisation in Theorem \ref{thm:classification-by-matrices}.
\begin{definition}
  \label{def:reduced-matrix}
  Let $\sigma \in \Sym(\ul m \times \ul n \times \ul o)$.  We identify $\sigma$ with its associated permutation matrix $(\sigma_{(i,j,k)(i', j', k')})_{(i,j,k), (i',j',k') \in \ul m \times \ul n \times \ul o}$ satisfying
  \begin{equation*}
    \sigma_{(i,j,k)(i', j', k')} = \delta_{\sigma(i',j',k'), (i,j,k)}
    \eqstop
  \end{equation*}
  The matrix $A = (A_{(i,k), (j',k')})_{(i,k) \in \ul m \times \ul o, (j',k') \in \ul n \times \ul o}$ with entries
  \begin{equation*}
    A_{(i,k), (j',k')} = \sum_{i' \in \ul m, j \in \ul n} \sigma_{(i,j,k)(i',j',k')}
  \end{equation*}
  is called the reduced matrix of $\sigma$.
\end{definition}

We start with a lemma that will be used in the proof of key Proposition \ref{prop:congruence-matrices}.
\begin{lemma}
  \label{lem:congruence-subpermutation-matrices}
  Let $C, D \in \rM_{m,n}(\{0,1\})$ be matrices such that in each row and each column there is at most one non-zero entry.  If $\sum_{i \in \ul m, j \in \ul n} C_{ij} = \sum_{i \in \ul m, j \in \ul n} D_{ij}$, then there are permutation matrices $\rho_1 \in \Sym(\ul m)$, $\rho_2 \in \Sym(\ul n)$ such that $C = \rho_1 D \rho_2$.
\end{lemma}
\begin{proof}
  First note that the assumptions imply that $C$ and $D$ have the same number of non-zero rows and non-zero columns.  We can hence find a permutation matrix $\rho_2 \in \Sym(\ul n)$ such that $C$ and $D \rho_2$ have the same non-zero columns.  Further, we can find a permutation matrix $\rho'_1 \in \Sym(\ul m)$ such that $C$ and $\rho'_1 D \rho_2$ have the same non-zero rows.  Since the non-zero columns of $\rho'_1 D \rho_2$ and $D \rho_2$ are the same, we infer that $C$ and $\rho'_1 D \rho_2$ have the same non-zero rows and non-zero columns.  Restricting to these non-zero rows and columns we obtain two $\{0,1\}$-valued matrices with exactly one non-zero entry in each row and each column - these are permutation matrices.  So we can replace $\rho'_1$ with some other $\rho_1 \in \Sym(\ul m)$ and obtain $C = \rho_1 D \rho_2$.
\end{proof}

\begin{proposition}
  \label{prop:congruence-matrices}
  Let $\sigma, \pi \in \Sym(\ul m \times \ul n \times \ul o)$. Let $A, B \in \rM_{\ul m \times \ul o , \, \ul n \times \ul o}(\NN)$ be the reduced matrices of $\sigma$ and $\pi$, respectively.  Then the following statements are equivalent.
  \begin{enumerate}
  \item \label{it:same-coset}
    $\sigma, \pi$ define the same element in
    $\Sym(\ul m \times \ul o) \lwreath \Sym(\ul n) \bs \Sym(\ul m \times \ul n \times \ul o) / \Sym(\ul m) \wreath \Sym(\ul n \times \ul o)$.
  \item \label{it:congruent-reduced-matrix}
    There are elements $\rho_1 \in \Sym(\ul m \times \ul o)$, $\rho_2 \in \Sym(\ul n \times \ul o)$ such that $\rho_1 A \rho_2 = B$.
  \end{enumerate}
\end{proposition}
\begin{proof}
  We first show that \ref{it:same-coset} implies \ref{it:congruent-reduced-matrix}.  Let 
  \begin{equation*}
    \rho \in \Sym(\ul m) \wreath \Sym(\ul n \times \ul o)
    =
    \bigl (\bigoplus_{\ul n \times \ul o} \Sym(\ul m) \bigr ) \rtimes \Sym(\ul n \times \ul o)
  \end{equation*}
  and denote by $C$ the reduced matrix of $\sigma \rho$.  Consider $\rho$ as an element in $\Sym(\ul m \times \ul n \times \ul o)$ and denote by $\rho_1: \ul m \times \ul n \times \ul o \ra \ul m$ and $\rho_2 \in \Sym(\ul n \times \ul o)$ the first and second coordinate of $\rho$ as described in Notation \ref{not:wreath-product}.  Recall that for fixed $(j,k) \in \ul n \times \ul o$, the map $\rho_1( \cdot, j,k): \ul m \ra \ul m$ is a bijection.  For all $i \in \ul m$, $j' \in \ul n$, $k, k' \in \ul o$
  \begin{align*}
    C_{(i,k),(j',k')}
    & =
      \sum_{i' \in \ul m, j \in \ul n} (\sigma \rho)_{(i,j,k)(i',j',k')}
    & \text{ (definition of reduced matrix) } \\
    & =
      \sum_{i' \in \ul m, j \in \ul n} (\sigma)_{(i,j,k) \rho(i',j',k')}
    & \text{ (multiplication of permutation matrices) } \\
    & =
      \sum_{i' \in \ul m, j \in \ul n} (\sigma)_{(i,j,k) (\rho_1(i',j',k'),\rho_2(j',k'))}
    & \text{ (splitting $\rho$ in coordinates) } \\
    & =
      \sum_{i' \in \ul m, j \in \ul n}
      (\sigma)_{(i,j,k) (i',\rho_2(j',k'))}
    & \text{ ($\rho_1( \cdot , j', k')$ is a bijection) } \\
    & =
      (A \rho_2)_{(i,k)(j',k')}
    & \text{ (multiplication with permutation matrices).}      
  \end{align*}
  So $C = A \rho_2$.  Similarly for $\rho \in \Sym(\ul m \times \ul o) \lwreath \Sym(\ul n)$ written in coordinates $\rho_1 \in \Sym(\ul m \times \ul o)$ and $\rho_2: \ul m \times \ul o \times \ul n \ra \ul n$, we obtain that the associated matrix of $\rho \sigma$ is $\rho_1 A$.  This shows the first implication.

  Let us now prove that \ref{it:congruent-reduced-matrix} implies \ref{it:same-coset}.  To this end, take $\sigma, \pi \in \Sym(\ul m \times \ul n \times \ul o)$ with reduced matrices $A$ and $B$, respectively and assume that there are $\rho_1 \in \Sym(\ul m \times \ul o)$ and $\rho_2 \in \Sym(\ul n \times \ul o)$ such that $\rho_1 A \rho_2 = B$.  By the previous paragraph, we can replace $\pi$ by $\rho_1^{-1} \pi \rho_2^{-1}$ and assume that $A = B$.  By the definition of reduced matrices, for all $(i,k) \in \ul m \times \ul o$ and all $(j',k') \in \ul n \times \ul o$ we have
  \begin{equation*}
    \sum_{i' \in \ul m, j \in \ul n} \sigma_{(i,j,k)(i',j',k')}
    =
    \sum_{i' \in \ul m, j \in \ul n} \pi_{(i,j,k)(i',j',k')}
    \eqstop
  \end{equation*}
  Put differently, the matrices $C, D \in \rM_{\ul n, \, \ul m}(\{0,1\})$ defined by
  \begin{equation*}
    C_{j i'} = \sigma_{(i,j,k)(i',j',k')}
    \qquad
    D_{j i'} = \pi_{(i,j,k)(i',j',k')}
  \end{equation*}
  have the same sum of their entries.  $C$ and $D$ are matrices with entries in $\{0, 1\}$ and on each row and on each column there is at most one non-zero entry.  Hence, Lemma \ref{lem:congruence-subpermutation-matrices} provides us with permutations $\rho_1( i, \cdot, k)  \in \Sym(\ul n)$ and $\rho_2( \cdot, j', k') \in \Sym(\ul m)$ such that $C = \rho_1 D \rho_2$.  Since $(i,k) \in \ul m \times \ul o$ and $(j', k') \in \ul n \times \ul o$ were arbitrary, we can put these permutations together and obtain elements $\rho_1 \in \bigoplus_{\ul m \times \ul o} \Sym(\ul n)$ and $\rho_2 \in \bigoplus_{\ul n \times \ul o} \Sym(\ul m)$ such that $\sigma = \rho_1 \pi \rho_2$.  This finishes the proof of the lemma.  
\end{proof}

\begin{lemma}
  \label{lem:matrix-flip}
  Let $\sigma \in \Sym(\ul m \times \ul m)$ with reduced matrix $A$.  Denote by $\nu \in \Sym( \ul m \times \ul m)$ the flip.  The reduced matrix of $\nu \sigma^{-1} \nu$ is $A^\rmt$.
\end{lemma}
\begin{proof}
  For all $i,j,k,l \in \ul m$, we have
  \begin{align*}
    (\nu \sigma^{-1} \nu)_{(i,j)(k,l)}
    & =
      \sigma^{-1}_{\nu^{-1}(i,j) \nu(k,l)}
    & \text{ (multiplication with permutation matrices) } \\
    & =
      \sigma^{-1}_{(j,i)(l,k)}
    & \text{ (definition of $\nu$) } \\
    & =
      \sigma_{(l,k)(j,i)}
      & \text{ (inverses of permutation matrices).}
  \end{align*}
  So the reduced matrix of $\nu \sigma^{-1} \nu$ satisfies
  \begin{equation*}
    \sum_{j, k \in \ul m} (\nu \sigma^{-1} \nu)_{(i,j)(k,l)}
    =
    \sum_{j, k \in \ul m}  \sigma_{(l,k)(j,i)}
    =
    A_{li}
    =
    A^\rmt_{il}
    \eqcomma
  \end{equation*}
  which proves the lemma.
\end{proof}

Let us give a name to the relation on $\rM_{m,n}(\NN)$ described by Proposition \ref{prop:congruence-matrices} and Lemma \ref{lem:matrix-flip}.  It does resemble, but is not exactly the same as usual notions of congruence: while on non-square matrices it agrees with a common notion of matrix congruence, we additionally call a square matrix and its transpose congruent.
\begin{definition}
  \label{def:congruence}
  Let $A, B \in \rM_{m,n}(\NN)$.  We call $A, B$ congruent to each other if there are permutations $\rho_1 \in \Sym(m)$, $\rho_2 \in \Sym(n)$ such that $\rho_1 A \rho_2 = B$ or $\rho_1 A^\rmt \rho_2 = B$, where the latter makes sense only in case $m = n$.
\end{definition}

\begin{theorem}
  \label{thm:cartan-classified-by-matrix}
  Let $B_\sigma, B_\pi \subset I_{m,n,o}$ be two twisted standard Cartan subalgebras.  Then $B_\sigma$ is conjugate to $B_\pi$ if and only if the reduced matrices of $\sigma$ and $\pi$ are congruent.
\end{theorem}
\begin{proof}
  This follows from combining Theorem \ref{thm:non-conjugate-cartans} with Proposition \ref{prop:congruence-matrices} and Lemma \ref{lem:matrix-flip}.
\end{proof}

In order to obtain a useful parametrisation of non-degenerate Cartan subalgebras, we have to identify those matrices that can arise as the reduced matrix of a permutation.  The following notation from matrix combinatorics allows us to describe them concisely.
\begin{notation}
  \label{not:row-column-sum-fixed}
  Let $a,b,c,d \in \NN$.  We let
  \begin{equation*}
    \rM(a,b,c,d)
    =
    \{ A \in \rM_{a,c}(\NN) \mid \forall i \in \{1, \dotsc, a\} : \,
    \sum_{j = 1}^c A_{ij} = b
    \text{ and }
    \forall j \in \{1, \dotsc, c\}: \, \sum_{i = 1}^a A_{ij} = d \}.
  \end{equation*}
\end{notation}

Theorem \ref{thm:cartan-classified-by-matrix} together with Theorem \ref{thm:classification-cartan-dimension-drop} associates with every non-degenerate Cartan subalgebra in $I_{m,n,o}$ the congruence class of a matrix in $\rM(mo, n, no, m)$.  The next proposition shows that this assignment is surjective.
\begin{proposition}
  \label{prop:characterise-reduced-matrices}
  The image of the map $\Sym(\ul m \times \ul n \times \ul o) \ra \rM_{\ul m \times \ul o, \, \ul n \times \ul o}(\NN)$ assigning to a permutation its reduced matrix is precisely $\rM(mo, n, no, m)$.

\end{proposition}
\begin{proof}
  It is clear that the reduced matrix of every element in $\Sym(\ul m \times \ul n \times \ul o)$ lies in $\rM(mo, n, no, m)$.  We have to show that every element of $\rM(mo, n, no, m)$ arises in this way. Let  $A \in \rM(mo, n, no, m)$. Define a matrix $\tilde A \in \rM_{\ul m \times \ul o, \, \ul n \times \ul o}(\rM_{n,m}(\{0,1\}))$ as follows. Fix $(i,k) \in \ul m \times \ul o$ and $(j',k') \in \ul n \times \ul o$. Let $M \in \rM_{n,  m}(\{0,1\})$ be the diagonal matrix whose first $A_{(i,k),(j',k')}$ diagonal entries are one and whose other entries are all zero. Equipping $\ul m \times \ul o$ and $\ul n \times \ul o$ with the respective lexicographical order, we moreover set
\begin{equation*}
r = \sum_{(t,u) < (i,k)} A_{(t,u),(j',k')} \quad \text{and} \quad s = \sum_{(t,u) < (j',k')} A_{(i,k),(t,u)}.
\end{equation*}
The $(i,k),(j',k')$-th entry of $\tilde A$ is now given as
\begin{equation*}
{\tilde A}_{(i,k),(j',k')} = (1 \, 2 \, \dotsm \, n)^r M (1 \, 2 \, \dotsm \, m)^s \eqcomma
\end{equation*}
where $(1 \, 2 \, \dotsm \, m) \in \Sym(m)$ and $(1 \, 2 \, \dotsm \, n) \in \Sym(n)$ are the respective full shift permutation matrices. Since the sum in every row of $A$ is $n$ and the sum of every column of $A$ is $m$, it follows that every row and every column of $\tilde A$, when interpreted as a matrix in $\rM_{\ul m \times \ul n \times \ul o}(\{0,1\})$, has exactly one non-zero entry. In other words, $\tilde A$ is a permutation matrix and thus corresponds to some $\pi \in \Sym(\ul m \times \ul n \times \ul o)$. By construction, $A$ is the reduced matrix of $\pi$. This concludes the proof.
\end{proof}

\begin{theorem}
  \label{thm:classification-by-matrices}
  Conjugacy classes of non-degenerate Cartan subalgebras in $I_{m,n,o}$ are parametrised by congruence classes of matrices in $\rM(mo, n, no, m)$.
\end{theorem}
\begin{proof}
  This follows from Theorems \ref{thm:classification-cartan-dimension-drop} and \ref{thm:cartan-classified-by-matrix} together with Proposition \ref{prop:characterise-reduced-matrices}.
\end{proof}

\subsection{The asymptotic number of Cartan subalgebras}
\label{sec:asymptotics}

Theorem \ref{thm:classification-by-matrices} in principle allows to apply results from enumerative combinatorics providing asymptotic formulae for the cardinality of $\rM(a,b,c,d)$ (see Notation \ref{not:row-column-sum-fixed}).  We refer to \cite{canfieldmckay10} and references therein for the reader who wants to know more about this topic.  However, there is no exact formula for the number of such matrices.  Further, the congruence relation introduced in Definition \ref{def:congruence} does not make part of the combinatorics literature, which obstructs a direct application.  Crude lower bounds on the number of congruence classes in $\rM(mo,n,no,m)$ can for example be given by appealing to the possible entries of a matrix in $\rM(mo,n,no,m)$ as a subset of $\{1, \dotsc, n\}$.  Despite this lower bound, it appears to be an interesting combinatorial problem to derive asymptotic formula for congruence classes in $\rM(a,b,c,d)$.

% obliges us to apply (crude) upper bounds on the number of matrices in a conjugacy class to obtain an estimates on the number of congruence classes in $\rM(mo,n, no, m)$.  This in turn allows estimate the number of pairwise non-conjugate Cartan subalgebras in dimension drop algebras.

% \begin{theorem}
%   \label{thm:estimata-cartan-subalgebras}
%   \comment{find an asymptotic lower bound on the number of Cartan subalgebras}
% \end{theorem}

\subsection{Explicit formula}
\label{sec:explicit-formula}

In Example \ref{ex:cartans-two-two}, we saw that the dimension drop algebra $I_{2,2}$ has exactly 2 Cartan subalgebras up to conjugacy.  This is the base case for two one-parameter series of dimension drop algebras that admit an explicit formula for the number of their Cartan subalgebras.  Both results have interesting consequences.

% We obtain Corollary \ref{cor:exactly-n-cartans}, providing the only examples of dimension drop algebras in which isomorphism and conjugacy of Cartan subalgebras is not the same. (See Section \ref{sec:homeomorphism} for our positive results on this problem).
 
The next proposition counts the number of Cartan subalgebras in stabilisations $I_{2,2,o}$ of $I_{2,2}$, making use of Theorem \ref{thm:classification-by-matrices}.   In Section \ref{sec:homeomorphism}, these results will provide examples of stabilised dimension drop algebras for which we will not be able to recover Cartan subalgebras up to conjugacy from the homeomophism type of their spectra.  We denote by $p$ the partition function, which counts the number of partitions of a non-negative integer.
\begin{proposition}
  \label{prop:cartans-two-o-two}
  In $I_{2,2,o}$ there are $p(2o)$ non-degenerate Cartan subalgebras up to conjugacy.
\end{proposition}
\begin{proof}
 By Theorem \ref{thm:classification-by-matrices}, it suffices to classify congruence classes of matrices in $\rM(2o,2,2o,2)$.  We will provide a symmetric normal form for congruence classes in $\rM(2o,2,2o,2)$ only using the $\Sym(2o) \times \Sym(2o)$ action.  A normal form is given by block diagonal matrices with the following blocks ordered by increasing size.
  \begin{equation*}
    \left (
      \begin{array}{c}
        2
      \end{array}
    \right ) \eqcomma
    \quad
    \left (
      \begin{array}{cc}
        1 & 1 \\
        1 & 1
      \end{array}
    \right ) \eqcomma
    \quad
    \left (
      \begin{array}{ccc}
        1 & 1 & 0 \\
        1 & 0 & 1 \\
        0 & 1 & 1
      \end{array}
    \right ) \eqcomma
    \quad
        \left (
      \begin{array}{cccc}
        1 & 1 & 0 & 0 \\
        1 & 0 & 1 & 0 \\
        0 & 1 & 0 & 1 \\
        0 & 0 & 1 & 1
      \end{array}
    \right ) \eqcomma \dotsc
  \end{equation*}
    Fix $o \in \NN$ and assume that we have normal forms in $\rM(2o', 2, 2o', 2)$ for $o' < o$.  Let $A \in \rM(2o,2,2o,2)$.  We may replace $A$ by a congruent matrix, so that $A_{11} \neq 0$ holds.  If $A_{11} = 2$, then it is the only non-zero entry in the first row and the first column of $A$, so $A$ is block diagonal and we may apply the induction hypothesis to the matrix obtained from $A$ by erasing the first row and the first column in order to obtain a normal form for $A$.  If $A_{11} = 1$, then we may further arrange for $A_{12} = 1 = A_{21}$ by passing to a congruent matrix.  Assume now that we arrived at a matrix congruent to $A$ which satisfies $A_{11} = A_{k,k - 1} = A_{k - 1, k} = 1$ for all $k \leq k_0$ for some $k_0 \geq 2$.  If $A_{k_0, k_0} = 1$, then $A$ is block diagonal and we obtain a normal form as before.  Otherwise, we can replace $A$ by a congruent matrix that satisfies $A_{k_0, k_0 + 1} = A_{k_0 + 1, k_0} = 1$.  This algorithm terminates and proves that $A$ is congruent to a normal form as described before.

    Next we show that the provided normal forms are pairwise non-congruent.  To this end, associate with a matrix in $\rM(2o,2, 2o,2)$ the graph whose vertices are indexed by $\ul{2o}\times \ul{2o}$ and whose edges are given by the rule $(i,j) \sim (i',j')$ if and only if the following two conditions are satisfied: $A_{ij} = 1 = A_{i'j'}$ and at the same time $i  = i'$ or $j = j'$.  Then the multiset of the size of connected components of this graph is a congruence invariant of $A$.  It distinguishes the normal forms provided before, because the block of size $k \times k$, $k \geq 2$ described above produces a single non-trivial connected component with $2k$ vertices.  Since the size of the blocks of our normal forms runs through all positive natural numbers and they have to fill the diagonal of a $2o \times 2o$ matrix, we conclude that there are $p(2o)$ congruence classes in $\rM(2o,2,2o,2)$.  This finishes the proof of the proposition.
\end{proof}

\begin{remark}
  \label{rem:not-isomorphism-conjugacy}
  Proposition \ref{prop:cartans-two-o-two} provides the exceptional examples of stabilised dimension drop algebras in which isomorphism and conjugacy of non-degenerate Cartan subalgebras is not the same.  This is formally stated in Proposition \ref{prop:no-conjugacy-homeomorphism}, which stands in contrast to the positive result provided by Theorem \ref{thm:conjugacy-homeomorphism}.
\end{remark}

The following proposition treats the other class of dimension drop algebras admitting an explicit computation of the number of their Cartan subalgebras.
\begin{proposition}
  \label{prop:cartans-two-n}
  The dimension drop algebra $I_{2,n}$ has exactly $\lfloor \frac{n}{2} \rfloor + 1$ non-degenerate Cartan subalgebras up to conjugacy.   In particular, if $n$ is odd, then $I_{2,n}$ has $\frac{n + 1}{2}$ Cartan subalgebras up to conjugacy.
\end{proposition}
\begin{proof}
  Example \ref{ex:cartans-two-two} showed that there are exactly 2 non-degenerate Cartan subalgebras in $I_{2,2}$, so that we may assume $n \geq 3$.  By Theorem \ref{thm:classification-by-matrices}, we have to count congruence classes in $\rM(2,n, n, 2)$.  We provide a normal form.  Let $A \in \rM(2,n, n, 2)$.  Let $k = |\{j \in \ul n \mid A_{1j} = 2\}|$.  Replacing $A$ by a congruent matrix, we may assume that $A_{11}, A_{12}, \dotsc, A_{1k} = 2$.  Since each column of $A$ sums to 2, all entries of $A$ are elements of $\{0,1,2\}$.  Moreover, each row of $A$ sums to $n$, so that we have $k = |\{j \in \ul n \mid A_{1j} = 0\}|$.  Replacing $A$ by a congruent matrix, we may assume that $A_{1, k+1}, A_{1, k+2}, \dotsc, A_{1,2k} = 0$.  Note that $A_{1, 2k + 1}, \dotsc, A_{1, n} = 1$ follows.  Moreover, $A_{2j} = 2 - A_{1j}$ for all $j \in \ul n$.  This is our normal form for $A$.  Two different normal forms are distinguished by $|\{(i,j) \in \ul 2 \times \ul n \mid A_{ij} = 2\}| \in 2\ZZ$,  which hence is a complete invariant for congruence classes in $\rM(2,n,n,2)$.  So indeed there are exactly $\lfloor \frac{n}{2} \rfloor + 1$ congruence classes.

  Let us now consider the case when $n$ is odd. In this case, Theorem \ref{thm:cartans-non-degenerate} shows that all Cartan subalgebras in $I_{2, n}$ are non-degenerate.  Further,
  \begin{equation*}
    \lfloor \frac{n}{2} \rfloor + 1
    =
    \frac{n + 1}{2}
    \eqcomma
  \end{equation*}
  which finishes the proof of the proposition.
\end{proof}

\begin{corollary}
  \label{cor:exactly-n-cartans}
    For every $n \in \NN$ there is a subhomogeneous \Cstar-algebra that has exactly $n$ Cartan subalgebras up to conjugacy.
\end{corollary}
\begin{proof}
  In \cite[Section 2.2]{lirenault17}, examples of homogeneous \Cstar-algebras without any Cartan subalgebras are presented. Furthermore, every homogeneous \Cstar-algebras over a contractible space has a unique Cartan subalgebra up to conjugacy; see Theorem~\ref{thm:li-renault}. For $n \geq 1$, the dimension drop algebra $I_{2, 2n+1}$ has exactly $n+1$ Cartan subalgebras up to conjugacy by Proposition \ref{prop:cartans-two-n}.
\end{proof}

\begin{remark}
  \label{rem:infinitely-many-cartans}
  Based on Corollary \ref{cor:exactly-n-cartans} it is possible to provide examples of \Cstar-algebras with exactly continuum many Cartan subalgebras up to conjugacy.  Let $A = \bigoplus_{n \geq 2} I_{2,2n + 1}$.  Every automorphism $\alpha \in \Aut(A)$ satisfies $\alpha(I_{2,2n + 1}) = I_{2,2n + 1}$ for all $n \geq 2$, since $I_{2,2n + 1}$ is $2n + 1$-subhomogeneous.  Further, it is easy to check that a Cartan subalgebra of $A$ is a direct sum of Cartan subalgebras of $I_{2,2n + 1}$ for $n \geq 2$.  Combining these two observations with Proposition \ref{prop:cartans-two-n}, we find that Cartan subalgebras of $A$ are parametrised by the product set $\prod_{n \geq 2} \ul{1 +  n}$ whose cardinality is the continuum.
\end{remark}

%%% Local Variables:
%%% mode: latex
%%% TeX-master: "cartan-dimension-drop-algebra"
%%% End:

%% file: homeomorphism.tex
\section{Isomorphism and conjugacy}
\label{sec:homeomorphism}

Our next aim is to show that although dimension drop algebras do not have a unique Cartan subalgebra up to conjugacy, often the next to best possible result is true.  Often Cartan subalgebras in dimension drop algebras are classified by their spectrum - which is their only intrinsic invariant.  We start by giving a concrete model for the spectrum of the twisted standard Cartan subalgebras considered in Example \ref{ex:twisted-standard-Cartan}
\begin{proposition}
  \label{prop:spectrum-twisted-cartan}
  Let $\sigma \in \Sym(\ul m \times \ul n \times \ul o)$. Then the spectrum of the twisted standard Cartan subalgebra $B_\sigma \subset I_{m,n,o}$ is homeomorphic with
  \begin{equation*}
    (\ul m \times \ul n \times \ul o) \times \bigl ( [0,1^-] \sqcup [1^+, 2] \bigr ) / \sim
    \eqcomma
  \end{equation*}
where the equivalence relation $\sim$ is given by the following three types of identifications.
\begin{align*}
  (i,j,k,0) & \sim (i,j',k, 0) 
  & \quad \text{ for all } i \in \ul m,  j, j' \in\ul n \text{ and } k \in \ul o \eqcomma  \\
  (i,j,k,2) & \sim (i',j,k, 2)
  & \quad \text{ for all } i,i' \in \ul m \text{ and } j \in \ul n, k \in \ul o \eqcomma \\
  (i,j,k,1^-) & \sim (\sigma(i,j,k), 1^+)
  & \quad \text{ for all } i \in \ul m, j \in \ul n \text{ and } k \in \ul o
    \eqstop
\end{align*}
\end{proposition}
\begin{proof}
  Let us first recall the definition of the twisted standard Cartan subalgebra associated with $\sigma$.  We fix a unitary
  \begin{equation*}
    u \in \cont([\frac{1}{3}, \frac{2}{3}], \rM_m \ot \rM_n \ot \rM_o)
  \end{equation*}
  such that $u_{\frac{1}{3}} = \sigma$ and $u_{\frac{2}{3}} = 1$.  Then 
  \begin{multline*}
    B_\sigma
    =
    \{
    f \in I_{m,n,o} \mid
    f(t) \in \rD_m \ot \rD_n \ot \rD_o \text{ for } t \in [0, \frac{1}{3}] \cup [\frac{2}{3}, 1] \\
    \text{ and }
    f(t) \in u_t(\rD_m \ot \rD_n \ot \rD_o)u_t^* \text{ for } t \in [\frac{1}{3}, \frac{2}{3}]
    \}
    \eqstop
  \end{multline*}
  The spectrum of $B_\sigma|_{[0, \frac{1}{3}]}$ is homeomorphic with $\ul m \times \ul n \times \ul o \times [0,1^-] / \sim$, with the identifications given by 
  \begin{align*}
    (i,j,k,0) & \sim (i,j',k, 0) 
    & \quad \text{ for all } i \in \ul m,  j, j' \in\ul n \text{ and } k \in \ul o
      \eqstop
  \end{align*}
  Similarly, the spectrum of $(B_\sigma)_{[\frac{2}{3}, 1]}$ is homeomorphic with $\ul m \times \ul n \times \ul o \times [1^+,2] / \sim$, with the identifications given by
  \begin{align*}
    (i,j,k,2) & \sim (i',j,k, 2)
    & \quad \text{ for all } i,i' \in \ul m \text{ and } j \in \ul n, k \in \ul o
    \eqstop
  \end{align*}
  These parts a glued by the spectrum of $(B_\sigma)_{[\frac{1}{3}, \frac{2}{3}]}$, providing the statement of the proposition.
\end{proof}

%\comment{insert picture of the spectrum of the standard Cartan}

\begin{notation}
  \label{not:spectrum-twisted-cartan}
  Given $\sigma \in \Sym(\ul m \times \ul n \times \ul o)$, we fix notation for certain points in the space
  \begin{equation*}
    (\ul m \times \ul n \times \ul o) \times \bigl ( [0,1^-] \sqcup [1^+, 2] \bigr ) / \sim
    \eqcomma
  \end{equation*}
  described in Proposition \ref{prop:spectrum-twisted-cartan}.  We denote by $(i, *, k, 0)$ the class of the points $(i,j,k,0)$ for $i \in \ul m$, $j \in \ul n$, $k \in \ul o$.  Analogously, we denote by $(*, j, k, 2)$ the class of the points $(i,j,k,2)$ for $i \in \ul m$, $j \in \ul n$, $k \in \ul o$.
\end{notation}

Having the model of Proposition \ref{prop:spectrum-twisted-cartan} at hand, we can describe the spectrum of a twisted standard Cartan subalgebra as the geometric realisation of a graph.  Let us fix the formalism of a graph that we use in the sequel.
\begin{notation}
  \label{not:graph}
  A graph $\Gamma$ is a triple $\Gamma = (\rV(\Gamma), \rE(\Gamma), a)$ with $\rV(\Gamma)$ interpreted as the set of vertices, $\rE(\Gamma)$ interpreted as the set of edges and $a: \rE(\Gamma) \ra \cP_{\leq 2}(\rV(\Gamma))$ the adjacency map.  We call 
 \begin{equation*}  
  |\{ e \in \rE(\Gamma) \mid v \in a(e)\}|
 \end{equation*}
 the valency of $v \in \rV(\Gamma)$. A graph all of whose vertices have the same valency $n \in \NN$ is called $n$-regular.

  We say that two vertices $v_1, v_2$ of $\Gamma$ are adjacent if there is an edge $e$ of $\Gamma$ such that $a(e) = \{v_1, v_2\}$.  A bi-partite graph is  a graph $\Gamma$ admitting a partition $\rV(\Gamma) = V_1 \sqcup V_2$ such that no vertex from $V_i$ is adjacent to a vertex of $V_i$, $i \in \{1, 2\}$.  If all vertices from $V_1$ have valency $m$ and all vertices from $V_2$ have valency $n$, we call $\Gamma$ an $(m,n)$-semi-regular bi-partite graph.  

  If $\Gamma$ is a finite bi-partite graph and $v_1, v_2, \dotsc, v_m, w_1, \dotsc, w_n$ is an enumeration of the vertices such that the sets $\{v_1, \dotsc, v_m\}$ and $\{w_1, \dotsc, w_n\}$ witness the fact that $\Gamma$ is bi-partite, then the adjacency matrix of $\Gamma$ is the $m \times n$ matrix whose $i,j$-th entry is the number of edges connecting $v_i$ and $w_j$.  Vice versa, if $A \in \rM_{m, n}(\NN)$, then the bi-partite graph associated with $A$ has vertices indexed by $(\ul m \times \{\rmr\}) \cup (\ul n \times \{\rmc\})$ and edges indexed by $(i,j, k)$ with $k \in \ul A_{ij}$ and $a(i,j,k) = \{(i, \rmr), (j, \rmc)\}$.
\end{notation}

\begin{notation}
  \label{not:geometric-realisation}
  Let $\Gamma$ be a graph as described in the formalism of Notation \ref{not:graph}.  We adopt the following notation for the geometric realisation $|\Gamma|$ of $\Gamma$.  It is the unique up to homeomorphism 1-dimensional CW-complex whose $0$-cells are indexed by $\rV(\Gamma)$ and whose 1-cells are indexed by $\rE(\Gamma)$, with the 1-cell of $e \in \rE(\Gamma)$ attached to $a(e)$.  Note that the latter could be a one-point set, in which case the 1-cell is glued to this single 0-cell, giving rise to a loop.
  % We fix an auxiliary order on $\rV(\Gamma)$.  We obtain from the adjacency map $a: \rE(\Gamma) \ra \cP_{\leq 2}(\rV(\Gamma))$ a map $\tilde a: a: \rE(\Gamma) \ra \rV(\Gamma)^2$ such that $\{\tilde a(e)_0, \tilde a(e)_1\} = a(e)$ and $\tilde a(e)_0 \leq \tilde a(e)_1$.  Note that $\tilde a (e)$ has two equal entries if and only if $a(e)$ is a one-point set.  We now set
  % \begin{equation*}
  %   |\Gamma|
  %   =
  %   \bigl ( \rV(\Gamma) \sqcup \rE(\Gamma) \times [0,1] \bigr ) / \sim
  % \end{equation*}
  % where the equivalence relation $\sim$ is given by
  % \begin{equation*} 
  %   (e, 0) \sim \tilde a(e)_0
  %   \qquad \text{ and } \qquad
  %   (e, 1) \sim \tilde a(e)_1
  %   \eqstop
  % \end{equation*}
  % It is easy to verify and well-known that $|\Gamma|$ does not depend on the chosen order on $\rV(\Gamma)$.
\end{notation}

\begin{proposition}
  \label{prop:spectrum-geometric realisation}
  Let $\sigma \in \Sym(\ul m \times \ul n \times \ul o)$.  Let $\Gamma$ be the bipartite graph associated with the reduced matrix of $\sigma$.  Then the spectrum of $B_\sigma$ is homeomorphic with the geometric realisation of $\Gamma$.
\end{proposition}
\begin{proof}
  We write $A$ for the reduced matrix of $\sigma$.  Denote by $X$ the spectrum of $B_\sigma$ as described in Proposition~\ref{prop:spectrum-twisted-cartan} and by $Y$ the geometric realisation of $\Gamma$.  In order to show that $X$ and $Y$ are homeomorphic, it suffices to obtain a description of $X$ as a CW-complex combinatorially isomorphic to the CW-complex described in Notation \ref{not:geometric-realisation}.

  Let us start by fixing the CW-structure on $X$.  The 0-cells of $X$ are the points
  \begin{equation*}
    \{ (i,*,k,0) \in X \mid i \in \ul m, k \in \ul o\}
    \cup
    \{(*,j,k, 2) \in X \mid j \in \ul n , k \in \ul o\}
    \eqstop
  \end{equation*}
  To define the 1-cells of $X$, recall that $X$ is a quotient of
  \begin{equation*}
    (\ul m \times \ul n \times \ul o) \times \bigl ( [0,1^-] \sqcup [1^+, 2] \bigr )  
  \end{equation*}
  with respect to an equivalence relation that identifies in particular the points $(i,j,k,1^-)$ and $(\sigma(i,j,k), 1^+)$ for all $i \in \ul m$, $j \in \ul n$ and $k \in \ul o$.  We take the 1-cells of $X$ to be the image of 
\begin{equation*}  
  \{(i,j,k)\} \times [0,1^-] \cup \{\sigma(i,j,k)\} \times [1^+, 2]
\end{equation*}  
 in $X$.  This way, 1-cells of $X$ are naturally indexed by $\ul m \times \ul n \times \ul o$.  Further, the 1-cell indexed by $(i,j,k)$ is glued to the 0-cells $(i,*,j)$ and $(* , \sigma(i,j,k)_2, \sigma(i,j,k)_3)$.

  Recall also the CW-structure on $Y$.  Its 0-cells are index by 
\begin{equation*}
\{(i,k,\rmr) \mid i \in \ul m , k \in \ul o\} \cup \{(j,k,\rmc) \mid j \in \ul n , k \in \ul o\} \eqcomma
\end{equation*}  
coming from the rows and columns of $A$.  There are $A_{(i,k)(j',k')}$ 1-cells glued between the 0-cells $(i,k, \rmr)$ and $(j',k', \rmc)$. 
  
  We can now establish a combinatorial isomorphism between the CW-complexes underlying $X$ and $Y$.  Fixing the natural bijection between the 0-cells of $X$
  \begin{equation*}
    \{ (i,*,k,0) \in X \mid i \in \ul m, k \in \ul o\}
    \cup
    \{(*,j,k, 2) \in X \mid j \in \ul n , k \in \ul o\}
  \end{equation*}
  and those of $Y$
  \begin{equation*}
    \{(i,k,\rmr) \mid i \in \ul m , k \in \ul o\}
    \cup
    \{(j,k,\rmc) \mid j \in \ul n , k \in \ul o\}
    \eqcomma
  \end{equation*}
  it suffices to show that the number of 1-cells glued between two 0-cells is preserved by this bijection.

  For $i \in \ul m$, $j' \in \ul n$ and $k,k' \in \ul o$, the 1-cells between $(i,*,k,0)$ and $(*,j',k', 2)$ are indexed by $(i,j,k)$, $j \in \ul n$ such that there is some $i' \in \ul m$ satisfying $(i,j,k) = \sigma(i',j',k')$.  Adopting the permutation matrix notation from Definition~\ref{def:reduced-matrix} for $\sigma$, this gives
  \begin{equation*}
    \sum_{i' \in \ul m, j \in \ul n} \sigma_{(i,j,k)(i',j',k')}
    =
    A_{(i,k)(j',k')}
  \end{equation*}
  many 1-cells.  This is what we had to show.  
\end{proof}

\begin{proposition}
  \label{prop:recover-graphs}
  Let $\Gamma$, $\Lambda$ be graphs with geometric realisations $X$ and $Y$, respectively.  Assume that either
  \begin{itemize}
  \item all vertices of $\Gamma$ and $\Lambda$ have valency at least 3, or
  \item that both $\Gamma$ and $\Lambda$ are bi-partite and $(2,n)$-semi-regular for some $n \geq 3$.
  \end{itemize}
  If $X \cong Y$, then $\Gamma \cong \Lambda$.
\end{proposition}
\begin{proof}
  We consider the space $X$ with its structure of a 1-dimensional CW-complex.  Set
  \begin{align*}
    \rV & = \{ x \in X \mid x \text{ is branch point}\} \\
    \rE & = \pi_0(X \setminus \rV)
          \eqcomma
  \end{align*}
  where a branch point is a point without any neighbourhood locally homeomorphic to an interval and $\pi_0$ denotes the set of connected components of a space.  If all vertices of $\Gamma$ have valency at least 3, then $\rV$ consists of the 0-cells of $X$.  Further for each connected component $e \in \pi_0(X \setminus V)$, the closure $\ol{e}$ is a 1-cell of $X$ and each 1-cell of $X$ arises uniquely in this way.  Hence, the homeomorphism type of $X$ recovers its CW-structure and therefore the isomorphism type of $\Gamma$.  If thus vertices of $\Lambda$ are also assumed to have valency at least 3, then $X \cong Y$ implies $\Gamma \cong \Lambda$.

  Assume now that $\Gamma$ and $\Lambda$ are bi-partite and $(2,n)$-semi-regular for some $n \geq 3$.  We associate with $\Gamma$ and $\Lambda$ the unique $n$-regular graphs $\tilde \Gamma$ and $\tilde \Lambda$ whose barycentric subdivision are $\Gamma$ and $\Lambda$, respectively.  Then $X$ is homeomorphic with the geometric realisation of $\tilde \Gamma$ and $Y$ is homeomorphic with the geometric realisation of $\tilde \Lambda$.  Hence, by the first statement of the proposition, $X \cong Y$ implies $\tilde \Gamma \cong \tilde \Lambda$.  This in turn implies $\Gamma \cong \Lambda$, which finishes the proof of the proposition.
\end{proof}

We are now ready to assemble the information obtained in this section so far and prove the following theorem.  It shows that in most stabilised dimension drop algebras conjugacy of non-degenerate Cartan subalgebras is equivalent to homeomorphism of their spectra.  However, the excluded cases $I_{2,2,o}$ for $o \geq 2$ do not satisfy this property, as we will see in Proposition \ref{prop:no-conjugacy-homeomorphism}.  We will make use of the next remark in the proof of Theorem \ref{thm:conjugacy-homeomorphism}.
\begin{remark}
  \label{rem:graph-adjacency-matrix}
 Let $A, B \in \rM_{m,n}(\NN)$ be two matrices and consider their associated bi-partite graphs $\Gamma, \Lambda$, respectively. Then $A$ is congruent to $B$ if and only if $\Gamma \cong \Lambda$.
\end{remark}

\begin{theorem}
  \label{thm:conjugacy-homeomorphism}
  Let $I_{m,n,o}$ be a stabilised dimension drop algebra such that either $(m,n) \neq (2,2)$ or $o = 1$.  Then two non-degenerate Cartan subalgebras of $I_{m,n,o}$ are conjugate by an automorphism if and only if their spectra are homeomorphic.
\end{theorem}
\begin{proof}
  Two conjugate non-degenerate Cartan subalgebras are isomorphic and hence have homeomorphic spectra.  We have to show the converse implication.  Let $(m,n,o) \in \NN_{\geq 1}^3$.  If $m = 1$ or $n = 1$, then Theorem \ref{prop:unique-cartan-one-sided-dimension-drop} shows that there is a unique non-degenerate Cartan subalgebra in $I_{m,n,o}$ up to conjugacy.  If $m = n = 2$ and $o =1$, then Example \ref{ex:cartans-two-two} shows that $I_{2,2,1} = I_{2,2}$ has exactly 2 non-degenerate Cartan subalgebras up to conjugacy and their spectra are not homeomorphic by Example \ref{ex:non-unique-Cartan}.

  It remains to treat the cases when $m, n \geq 2$ and $m \geq 3$ or $n \geq 3$.  By Theorem \ref{thm:classification-by-matrices}, non-degenerate Cartan subalgebras of $I_{m,n,o}$ are classified by congruence classes of matrices in $\rM(mo,n,no, m)$.   From Proposition \ref{prop:spectrum-geometric realisation} we know that the spectrum of a non-degenerate Cartan algebra $B$ is the geometric realisation of the bi-partite graph $\Gamma$ associated with some $A \in \rM(mo,n,no,m)$.  Thanks to our assumption on $m$ and $n$, Proposition~\ref{prop:recover-graphs} applies and yields that the homeomorphism type of the spectrum of $B$ recovers the isomorphism class of $\Gamma$.  Now we can recover the congruence class of $A$ by Remark \ref{rem:graph-adjacency-matrix}, which finishes the proof of the theorem.
  \end{proof}

We finish this article, by pointing out that the conclusion of Theorem \ref{thm:conjugacy-homeomorphism} does not hold for the dimension drop algebras omitted from its statement.
\begin{proposition}
  \label{prop:no-conjugacy-homeomorphism}
  Let $o \geq 2$.  In $I_{2,2,o}$ there are non-conjugate non-degenerate Cartan subalgebras with homeomorphic spectra.  More precisely, $I_{2,2,o}$ admits exactly $p(2o)$ pairwise non-conjugate non-degenerate Cartan subalgebras and their spectra assume exactly $2o$ different homeomorphism types.
\end{proposition}
\begin{proof}
  Proposition \ref{prop:cartans-two-o-two} says that there are $p(2o)$ non-degenerate Cartan subalgebras in $I_{2,2,o}$ up to conjugacy.  At the same time, it follows from Proposition \ref{prop:spectrum-twisted-cartan} that the spectrum of a non-degenerate Cartan subalgebra of $I_{2,2,o}$ is the geometric realisation of a 2-regular bi-partite graph with $2o$ vertices of each kind.  It follows that its spectrum is a disjoint union of up to $2o$ circles.  It remains to check that for each $1 \leq k \leq 2o$ there is a non-degenerate Cartan subalgebra of $I_{2,2,o}$ whose spectrum is homeomorphic with a disjoint union of exactly $k$ circles.  Such is associated -- following the terminology of the proof of Proposition \ref{prop:cartans-two-o-two} -- with a block diagonal matrix in $\rM(2o, 2, 2o, 2)$ with exactly $k$ blocks taken from
  \begin{equation*}
    \left (
      \begin{array}{c}
        2
      \end{array}
    \right ) \eqcomma
    \quad
    \left (
      \begin{array}{cc}
        1 & 1 \\
        1 & 1
      \end{array}
    \right ) \eqcomma
    \quad
    \left (
      \begin{array}{ccc}
        1 & 1 & 0 \\
        1 & 0 & 1 \\
        0 & 1 & 1
      \end{array}
    \right ) \eqcomma
    \quad
        \left (
      \begin{array}{cccc}
        1 & 1 & 0 & 0 \\
        1 & 0 & 1 & 0 \\
        0 & 1 & 0 & 1 \\
        0 & 0 & 1 & 1
      \end{array}
    \right ) \eqcomma \dotsc
  \end{equation*}
  Since $p(2o) > 2o$ for $o \geq 2$, the proof is finished.
\end{proof}

%%% Local Variables:
%%% mode: latex
%%% TeX-master: "cartan-dimension-drop-algebra"
%%% End:

%% file: cartan-dimension-drop-algebra.bbl
\begin{thebibliography}{10}

\bibitem{anderson79}
J.~Anderson.
\newblock {Extensions, restrictions, and representations of states on
  C$^*$-algebras.}
\newblock {\em Trans. Amer. Math. Soc.}, 249(2):303–329, 1979.

\bibitem{barlakli15-uct}
S.~Barlak and X.~Li.
\newblock {Cartan subalgebras and the UCT problem.}
\newblock {\em Adv. Math.}, 316:748--769, 2017.

\bibitem{barlakli17-uct}
S.~Barlak and X.~Li.
\newblock {Cartan subalgebras and the UCT problem, II.}
\newblock arXiv:1704.04939, 2017.

\bibitem{barlakszabo17-problems}
S.~Barlak and G.~Szab{\'o}.
\newblock Problem sessions.
\newblock In {\em Oberwolfach Report}, volume~42, pages 26 -- 27. 2017.

\bibitem{berbecvaes12}
M.~Berbec and S.~Vaes.
\newblock {W$^*$-superrigidity for group von Neumann algebras of left-right
  wreath products.}
\newblock {\em Proc. Lond. Math. Soc.}, 108(5):1116--1152, 2014.

\bibitem{blackadar06}
B.~Blackadar.
\newblock {\em {Operator algebras. Theory of $\mathrm{C}^*$-algebras and von
  Neumann algebras.}}, volume 122 of {\em Encyclopaedia of Mathematical
  Sciences. Operator Algebras and Non-Commutative Geometry III.}
\newblock Berlin-Heidelberg: Springer-Verlag, 2006.

\bibitem{canfieldmckay10}
E.~R. Canfield and B.~D. McKay.
\newblock {Asymptotic enumeration of integer matrices with large equal row and
  column sums.}
\newblock {\em Combinatorica}, 30(6):655--680, 2010.

\bibitem{connesjones85}
A.~Connes and V.~F.~R. Jones.
\newblock {Property (T) for von Neumann algebras.}
\newblock {\em Bull. Lond. Math. Soc.}, 17:57--62, 1985.

\bibitem{deeleyputnamstrung2015}
R.~J. Deeley, I.~F. Putnam, and K.~R. Strung.
\newblock {Constructing minimal homeomorphisms on point-like spaces and a
  dynamical presentation of the Jiang-Su algebra.}
\newblock {\em J. Reine Angew. Math.}, 2015.

\bibitem{dixmier53-sous-anneaux-abliens}
J.~Dixmier.
\newblock {Sous-anneaux ab{\'e}liens maximaux dans les facteurs de type fini.}
\newblock {\em Ann. of Math. (2)}, 59(2):279--286, 1954.

\bibitem{elliottgonglinniu2015}
G.~A. Elliott, G.~Gong, H.~Lin, and Z.~Niu.
\newblock {On the classification of simple amenable C$^*$-algebras with finite
  decomposition rank, II.}
\newblock arXiv:1507.03437, 2015.

\bibitem{feldmanmoore77}
J.~Feldman and C.~C. Moore.
\newblock {Ergodic equivalence relations, cohomology, and von Neumann algebras.
  I.}
\newblock {\em Trans. Am. Math. Soc.}, 234:289--324, 1977.

\bibitem{feldmanmoore77_1}
J.~Feldman and C.~C. Moore.
\newblock {Ergodic equivalence relations, cohomology, and von Neumann algebras.
  II.}
\newblock {\em Trans. Am. Math. Soc.}, 234:325--359, 1977.

\bibitem{gonglinniu15}
G.~Guihua, H.~Lin, and Z.~Niu.
\newblock {Classification of finite simple amenable $\mathcal Z$-stable
  C$^*$-algebras.}
\newblock arXiv:1501.00135, 2015.

\bibitem{ioanapopavaes10}
A.~Ioana, S.~Popa, and S.~Vaes.
\newblock {A class of superrigid group von Neumann algebras.}
\newblock {\em Ann. Math. (2)}, 178:231--286, 2013.

\bibitem{jiangsu99}
X.~Jiang and H.~Su.
\newblock {On a simple unital projectionless C$^*$-algebra.}
\newblock {\em Am. J. Math.}, 121:359--413, 1999.

\bibitem{katsura2008}
T.~Katsura.
\newblock {A class of C$^*$-algebras generalizing both graph algebras and
  homeomorphism C$^*$-algebras IV, pure infiniteness.}
\newblock {\em J. Funct. Anal.}, 254:1161--1187, 2008.

\bibitem{kundbyraumthielwhite17}
S.~Knudby, S.~Raum, H.~Thiel, and S.~White.
\newblock {On C$^*$-superrigidity of virtually abelian groups.}
\newblock In preparation.

\bibitem{krogagervaes15}
A.~S. Krogager and S.~Vaes.
\newblock {A class of \textrm{II}$_1$ factors with exactly two crossed product
  decompositions.}
\newblock {\em J. Math. Pures Appl.}, 108(1):88--110, 2017.

\bibitem{kumjian86}
A.~Kumjian.
\newblock {On C$^*$-diagonals.}
\newblock {\em Can. J. Math.}, 38(4):969--1008, 1986.

\bibitem{li15-rigidity}
X.~Li.
\newblock {Continuous orbit equivalence rigidity.}
\newblock arXiv:1503.01704, accepted for publication in Erdog. Th. Dyn. Sys.,
  2015.

\bibitem{li16-quasi-isometry}
X.~Li.
\newblock {Dynamic characterizations of quasi-isometry, and applications to
  cohomology.}
\newblock arXiv:1604.07375, 2016.

\bibitem{lirenault17}
X.~Li and J.~Renault.
\newblock {Cartan subalgebras in C$^*$-algebras. Existence and uniqueness.}
\newblock arXiv:1703.10505, 2017.

\bibitem{matuisato2012}
H.~Matui and Y.~Sato.
\newblock {Strict comparison and $\mathcal Z$-absorption of nuclear
  C$^*$-algebras.}
\newblock {\em Acta Math.}, 209:179--196, 2012.

\bibitem{matuisato2014}
H.~Matui and Y.~Sato.
\newblock {Decomposition rank of UHF-absorbing C$^*$-algebras.}
\newblock {\em Duke Math. J.}, 163:2687--2708, 2014.

\bibitem{ozawapopa10-cartan1}
N.~Ozawa and S.~Popa.
\newblock {On a class of II$_1$ factors with at most one Cartan subalgebra.}
\newblock {\em Ann. Math. (2)}, 172(1):713--749, 2010.

\bibitem{popa83-orthogonal}
S.~Popa.
\newblock {Orthogonal pairs of $*$-subalgebras in finite von Neumann algebras.}
\newblock {\em J. Oper. Theory}, 9(2):253--268, 1983.

\bibitem{popa06-betti-numbers}
S.~Popa.
\newblock {On a class of type II$_1$ factors with Betti numbers invariants.}
\newblock {\em Ann. Math. (2)}, 163(3):809--899, 2006.

\bibitem{popa06-strong-rigidity-1}
S.~Popa.
\newblock {Strong rigidity of II$_1$ factors arising from malleable actions of
  $w$-rigid groups. I.}
\newblock {\em Invent. Math.}, 165(2):369--408, 2006.

\bibitem{popavaes10-superrigidity}
S.~Popa and S.~Vaes.
\newblock {Group measure space decomposition of II$_1$ factors and
  W$^{*}$-superrigidity.}
\newblock {\em Invent. Math.}, 182(2):371--417, 2010.

\bibitem{popavaes11_2}
S.~Popa and S.~Vaes.
\newblock {Unique Cartan decomposition for \textrm{II}$_1$ factors arising from
  arbitrary actions of free groups.}
\newblock {\em Acta Math.}, 212:141--198, 2014.

\bibitem{popavaes12}
S.~Popa and S.~Vaes.
\newblock {Unique Cartan decomposition for \textrm{II}$_1$ factors arising from
  arbitrary actions of hyperbolic groups.}
\newblock {\em J. Reine Angew. Math.}, 694:215--239, 2014.

\bibitem{putnam2016}
I.~F. Putnam.
\newblock {Some classifiable groupoid C$^*$-algebras with prescribed K-theory.}
\newblock arXiv:1611.04649, accepted for publication in Math. Ann., 2016.

\bibitem{renault80}
J.~Renault.
\newblock {\em {A groupoid approach to C$^*$-algebras.}}, volume 793 of {\em
  Lecture Notes in Mathematics}.
\newblock Berlin: Springer-Verlag, 1980.

\bibitem{renault08-cartan}
J.~Renault.
\newblock {Cartan subalgebras in C$^*$-algebras.}
\newblock {\em Irish Math. Soc. Bulletin}, 61:29--63, 2008.

\bibitem{satowhitewinter2015}
Y.~Sato, S.~White, and W.~Winter.
\newblock {Nuclear dimension and $\mathcal Z$-stability.}
\newblock {\em Invent. Math}, 202:893--921, 2015.

\bibitem{singer55}
I.~M. Singer.
\newblock {Automorphisms of finite factors.}
\newblock {\em Am. J. Math.}, 77:117--133, 1955.

\bibitem{spaas17}
P.~Spaas.
\newblock {Non-classification of Cartan subalgebras for a class of von Neumann
  algebras.}
\newblock arXiv:1710.10771.

\bibitem{spakulawillett11}
J.~Spakula and R.~Willett.
\newblock On rigidity of roe algebras.
\newblock {\em Adv. Math.}, 249:289--310, 2013.

\bibitem{speelmanvaes12}
A.~Speelman and S.~Vaes.
\newblock {A class of \textrm{II}$_{1}$ factors with many non conjugate Cartan
  subalgebras.}
\newblock {\em Adv. Math.}, 231(3-4):2224--2251, 2012.

\bibitem{spielberg07}
J.~Spielberg.
\newblock {Graph-based models for Kirchberg algebras.}
\newblock {\em J. Oper. Theory}, 57:347--374, 2007.

\bibitem{tikuisiswhitewinter17}
A.~Tikuisis, S.~White, and W.~Winter.
\newblock {Quasidiagonality of nuclear C$^*$-algebras.}
\newblock {\em Ann. of Math. (2)}, 185:229--284, 2017.

\bibitem{yeend06}
T.~Yeend.
\newblock {Topological higher-rank graphs and the C$^*$-algebras of topological
  1-graphs.}
\newblock In {\em Operator theory, operator algebras, and applications}, volume
  414 of {\em Contemp. Math.}, pages 231--244. Providence, RI: American
  Mathematical Society, 2006.

\bibitem{yeend07}
T.~Yeend.
\newblock {Groupoid models for the C$^*$-algebras of topological higher-rank
  graphs.}
\newblock {\em J. Oper. Theroy}, 57:95--120, 2007.

\end{thebibliography}
